\documentclass[12pt]{amsart}
\usepackage{stmaryrd}
\allowdisplaybreaks

\usepackage[bookmarks=true, bookmarksopen=true,%
bookmarksdepth=3,bookmarksopenlevel=2,%
colorlinks=true,%
linkcolor=blue,%
citecolor=blue,%
filecolor=blue,%
menucolor=blue,%
urlcolor=blue]{hyperref}

\usepackage{anysize}
\marginsize{1in}{1in}{1in}{1in}

\usepackage{amssymb}
\usepackage{amsbsy}
\usepackage{amscd}
\usepackage[mathscr]{eucal}
\usepackage{verbatim}
\usepackage{version}
\usepackage{graphicx}

\usepackage{enumerate}

\usepackage{color}

\definecolor{darkgreen}{rgb}{0.0, 0.7, 0.0}
\definecolor{darkpurple}{rgb}{0.7, 0.0, 0.7}

\def\lbr{\llbracket}
\def\rbr{\rrbracket}

\newenvironment{pf}{\proof[\proofname]}{\endproof}
\newenvironment{pf*}[1]{\proof[#1]}{\endproof}

\newtheorem{thm}{Theorem}

\newtheorem{cor}[thm]{Corollary}
\newtheorem{prop}[thm]{Proposition}
\newtheorem{conj}[thm]{Conjecture}

\theoremstyle{definition}

\newtheorem{defn}[thm]{Definition}
\newtheorem{rem}[thm]{Remark}

\newtheorem{nota}[thm]{Notation}
\newtheorem{example}[thm]{Example}

\numberwithin{equation}{section}

\newcommand{\thmref}[1]{Theorem~\ref{#1}}

\newcommand{\propref}[1]{Proposition~\ref{#1}}
\newcommand{\corref}[1]{Corollary~\ref{#1}}

\newcommand{\defref}[1]{Definition~\ref{#1}}
\newcommand{\remref}[1]{Remark~\ref{#1}}
\newcommand{\conjref}[1]{Conjecture~\ref{#1}}

\newcommand{\C}{{\Bbb C}}
\newcommand{\Z}{{\Bbb Z}}
\newcommand{\Q}{{\Bbb Q}}
\newcommand{\R}{{\Bbb R}}

\newcommand{\D}{{\mathcal D}}

\newcommand{\Coeff}{\mathop{\text{\rm Coeff}}}
\def\oo{{\cal O}}

\def\P{{\mathbb P}}
\def\<{\langle}
\def\>{\rangle}

\def\FD{{\mathbf{FD}}}
\def\trop{{\text{trop}}}
\def\R{{\mathbb R}}

\def\Q{{\mathbb Q}}

\def\P{{\mathbb P}}
\def\Z{{\mathbb Z}}
\def\C{{\mathbb C}}

\def\oo{{\mathcal O}}

\newcommand{\sgn}{\mathop{\text{\rm sgn}}}

\def\={\;=\;}

\def\+{\,+\,} \def\-{\,-\,}
\def\be{\begin{equation}}
\def\ee{\end{equation}}
\def\bes{\begin{equation*}}
\def\ees{\end{equation*}}

\def\mult{\text{mult}}

\def\minv{{\rm minv}}
\def\maxv{{\rm maxv}}
\def\ext{{\rm ext}}

\begin{document}
\title[Refined node polynomials via long edge graphs]{Refined node polynomials via long edge graphs}

\author{Lothar G\"ottsche \and Benjamin Kikwai}

\date\today

\begin{abstract} The generating functions of the Severi degrees for sufficiently ample line bundles on algebraic surfaces  are multiplicative in the
topological invariants of the surface and the line bundle. Recently new proofs of this fact were given for toric surfaces by Block, Colley, Kennedy and Liu, Osserman, using tropical geometry and in particular the   combinatorial tool of 
long-edged graphs. In the first part of this paper these results are for $\P^2$ and rational ruled surfaces generalised to refined Severi degrees.
In the second part of the paper we give a number of mostly conjectural generalisations of this result to singular surfaces, and curves with prescribed multiple points. 
	\end{abstract}

\maketitle
\section{Introduction}

The  Severi degree $n^{d,\delta}$ is the number of $\delta$-nodal degree $d$ curves in the projective plane $\P^2$ through 
$d(d+3)/2-\delta$ general points. More  generally for a pair $(S,L)$ of a complex projective surface a line bundle on $S$, the 
 Severi degree $n^{(S,L),\delta}$ counts the number of $\delta$-nodal curves in the linear system $|L|$ passing through $\dim|L|-\delta$ general points. 
In \cite{dFI} is was conjectured that there are polynomials   $n_{\delta}(d)$ in $d$, called {\em node polynomials}, such  that $n^{d,\delta}=n_\delta(d)$, for  $d$ sufficiently large with respect to $\delta$.
In \cite{G} it was conjectured that there are universal polynomials
$t_{\delta}(x,y,z,w)$, such that  for $L$ sufficiently ample with respect to $\delta$, $n^{(S,L),\delta}$ is  obtained by substituting the intersection numbers
$L^2$, $LK_S$, $K_S^2$, $\chi(\oo_S)$: writing $n_\delta{(S,L)}:=t_\delta(L^2, LK_S,K_S^2,\chi(\oo_S))$ we should have 
$n_\delta(S,L)=n^{(S,L),\delta}$. 
The conjectures of \cite{G} furthermore express the generating functions
$$n(d):=\sum_{\delta\ge 0} n_\delta(d)t^\delta,\qquad n(S,L):=\sum_{\delta\ge 0} n_\delta(S,L)t^\delta$$ in terms of some universal power series.
$n(S,L)$ is multiplicative in the parameters, i.e. 
\begin{equation}\label{prodSL}
n(S,L)=A_1(t)^{L^2}A_2(t)^{LK_S}A_3(t)^{K_S^2}A_4(t)^{\chi(\oo_S)},
\end{equation} for some power series 
$A_i(t)\in \Q[[t]]$, and thus in particular
\begin{equation}\label{prodP2}
n(d)=A_1(t)^{d^2}A_2(t)^{-3d}A_3(t)^{9}A_4(t).
\end{equation}
Furthermore explicit formulas for $A_1(t)$ and $A_4(t)$ are given in terms of modular forms.

We will call \eqref{prodSL} and \eqref{prodP2} the {\it multiplicativity }of $n(S,L)$ and $n(d)$.
The Severi degrees of $\P^2$ and toric surfaces can be computed via tropical geometry, by the Mikhalkin correspondence theorem \cite{Mik}.
This was used in \cite{FM} to prove the existence of the node polynomials $n_\delta(d)$, using Floor diagrams which are combinatorial devices for encoding tropical curves.
The conjecture of \cite{G} was proven in \cite{Tz}, \cite{KST}, using the methods of complex geometry. 
 In \cite{BCK} and \cite{L} the Severi degrees $n^{d,\delta}$ were studied using long edge graphs, a modification of floor diagrams, giving  an alternative proof for the
 multiplicativity of the generating function $n(d)$.  This is done by taking the formal logarithm.
  $$Q(d):=\log(n(d))=\sum_{\delta\ge 1} Q_\delta(d) t^\delta.$$
 The multiplicativity for $n(d)$ is equivalent to the statement that $Q_\delta(d)$ is a polynomial of degree $2$ in $d$ for all $\delta$. This is proven in \cite{BCK} and \cite{L}, giving the first purely combinatorial proofs of  \eqref{prodP2}.
 In \cite{LO} this result is generalized to a large class of toric surfaces, and a generalisation is given to toric surfaces with rational singularities.
 This note tries to  extend these results to the refined Severi degrees defined in \cite{GS} and \cite{BG} and thus also to the Welschinger numbers.
  
The Welschinger numbers $W^{d,\delta}$ count $\delta$-nodal degree $d$ real curves  in $\P^2$ through $d(d+3)/2-\delta$ real points with suitable signs, and 
 $W^{(S,L),\delta}$ counts real $\delta$-nodal curves in the linear system $|L|$ on a real algebraic surface $S$ through a configuration of $\dim |L|-\delta$ real points.
 They are closely related to to the Welschinger invariants, deformation invariants defined in genus $0$.
 The Welschinger numbers depend in general on the point configuration, but in \cite{Mik} it is shown that, for a so called {\em subtropical} configuration of points, they coincide with the tropical Welschinger invariants $W^{\trop}_{d,\delta}$, $W^\trop_{(S,L),\delta}$, defined via tropical geometry (and these are independent of the tropical configuration of points). 
 In future we will assume that we are dealing with a subtropical configuation of points .
 
 In \cite{GS} and \cite{BG} refined Severi degrees $N^{d,\delta}(y)$, and $N^{(S,L),\delta}(y)$ for toric surfaces are introduced via tropical geometry. 
 These  are symmetric Laurent polynomials 
 in a variable $y$,  interpolating between the Severi degrees  and the Welschinger numbers, i.e.  $N^{(S,L),\delta}(1)=n^{(S,L),\delta}$, $N^{(S,L),\delta}(-1)=W^{(S,L),\delta}$.
 In \cite{GS}  analogues of the conjectures of \cite{G} are formulated for the refined Severi degrees. 
 In particular for $\delta\le 2d-2$  the  $N^{d,\delta}(y)$ should be given by refined node polynomials $N_\delta(d;y)\in \Q[d,y^{\pm1}]$.
 Similarly for pairs $(S,L)$ of a  smooth toric surface and a $\delta$-very ample toric line bundle the conjectures say
 $N^{(S,L),\delta}(y)=N_\delta((S,L);y)$, for some polynomial in $N_\delta((S,L);y)$ in  $L^2, LK_S, K_S^2, \chi(\oo_S)$. In the case of $\P^2$, $\P(1,1,m)$ or a Hirzebruch surface $\Sigma_m$, 
 these conjectures are (with  weaker bounds) proven in \cite[Thm.~4.2]{BG}.

 We  introduce  generating functions for the refined node polynomials.
 Let 
 $$N(d)(y,t):=\sum_{\delta\ge 0} N_\delta(d;y)t^\delta, \qquad N(S,L)(y,t):=\sum_{\delta\ge 0} N_\delta (S,L;y)t^\delta.$$
 In \cite{GS} it is again conjectured that  $N(S,L)(y,t)$  is multiplicative.
 
 \begin{conj}\label{RefMult} \cite{GS}
 There exist power series $A_i(y,t)\in \Q[y^{\pm 1}][[t]], \ i=1,2,3,4$, such that for all pairs $(S,L)$ of a smooth toric surface and a toric line bundle we have
 \begin{equation}\label{prodSLy}
 \begin{split}
N(S,L)(y,t)=A_1(y,t)^{L^2}A_2(y,t)^{LK_S}A_3(t)^{K_S^2}A_4(t)^{\chi(\oo_S)},\\
N(d)(y,t)=A_1(y,t)^{d^2}A_2(y,t)^{-3d}A_3(y,t)^{9}A_4(y,t).
\end{split}
\end{equation}
\end{conj}
Two of these  power series are expressed in terms of Jacobi forms. One can rewrite \eqref{prodSLy} in a different way:
\begin{equation} \label{othergen} N_\delta (S,L;y)=\Coeff_{q^{L(L-K_S)/2}}\Big[DG_2(y,q)^{L(L-K_S)/2-\delta}B_1(y,q)^{K_S^2}B_2(y,q)^{LK_S}B_3(y,q)^{\chi(\oo_S)}\Big]
\end{equation} 
Here $DG_2(y,q),B_i(y,q)\in \Q[y^{\pm 1}][[q]]$, and $DG_2(y,q)$, $B_3(y,q)$ are related to theta functions.
For more details see Section 4.


In the first part of the current note we  adapt the method of long edge graphs and  the proofs of \cite{BCK}, \cite{L}, \cite{LO} to refined Severi degrees, to prove the multiplicativity 
also for the  $N(d)(y,t)$ and a weaker version of multiplicativity for rational ruled surfaces
(see \thmref{mainthm}).
We combine this with computer calculations of  the refined Severi degrees and the Welschinger numbers of $\P^2$ and rational ruled surfaces.
This allows to determine the refined node polynomials of $\P^2$ and rational ruled surfaces for low values $\delta$, confirming the predictions of \cite{GS}
(see \corref{refpol}),
and   extending  the  results of 
\cite{BG}.

We then extend the results and conjectures to surfaces with singularities and to curves passing through (smooth or singular) points of $S$ with higher multiplicity.
This in particular includes a conjectural generalisation of the results of \cite{LO} to the refined invariants.
The conjectural formulas  generalize \eqref{othergen}.
For every condition $c$ that we can impose on the curves at a point of $S$, we get  a power series
$D_c(y,q)\in \Q[y^{\pm 1}][[q]]$, such that the refined count of curves in $|L|$ on $S$ satisfying conditions $c_1,\ldots,c_s$ will be given by
\begin{equation} \label{othergen1} \Coeff_{q^{L(L-K_S)/2}}\Big[ B_1(y,q)^{K_S^2}B_2(y,q)^{LK_S}B_3(y,q)^{\chi(\oo_S)}\prod_{i=1}^s D_{c_i}(y,q)\Big]
\end{equation} 
The formula \eqref{othergen} is the case that the conditions imposed are to pass through $L(L-K_S)/2-\delta$ general points, in particular 
$DG_2(y,q)$ is the power series corresponding to the condition of passing through a point. 

 \section{Refined Severi degrees and long edge graphs}
 \subsection{Refined Severi degrees and Floor diagrams}
\label{subsecRefinedSeveri}

In \cite{GS}, \cite{BG} refined Severi degrees were introduced. We will briefly recall some of the results and definitions. 

A lattice  polygon $\Delta\subset \R^2$ is a polygon
with vertices of
 integer coordiates.
The lattice length of an edge $e$ of $\Delta$ is $\#( e\cap \Z^2)-1$. 
We denote by $int(\Delta)$, $\partial(\Delta)$ its interior and its boundary.
To a convex lattice polygon $\Delta$ one can associate a pair $S(\Delta)$, $L(\Delta)$ of a toric surface and 
a toric line bundle on $S(\Delta)$.  The toric surface is defined by the fan given by the outer normal vectors of $\Delta$. We have
$\dim H^0(S(\Delta),L(\Delta))=\# (\Delta\cap \Z^2)$. The arithmetic genus of a curve in $|L(\Delta)|$ is 
$g(\Delta)=\# (int(\Delta)\cap \Z^2)$.
In \cite[Def.~3.8]{BG}  refined Severi degrees $N^{\Delta,\delta}(y)$ are defined for any convex lattice polygon $\Delta$. They are a count of tropical curves in $\R^2$ satisfying suitable point conditions with multiplicities which are Laurent polynomials in $y$. We also write $N^{S(\Delta),L(\Delta),\delta}(y):=N^{\Delta,\delta}(y)$.
The $N^{\Delta,\delta}(y)$ interpolate between the Severi degrees (at $y=1$) and the tropical Welschinger numbers (at $y=-1$).

\begin{example}\label{polygons}
In the following we will be concerned only with the following lattice polygons
$\Delta_{c,m,d}=\big\{(x,y)\in (\R_{\ge 0}\bigm| y\le d; \ x+my\le md+c\big\}$, for $d,\ge 0, m\ge 0,c\ge 0$.
These   are so called $h$-transversal lattice polygons, i.e. all the slopes of the outer normal vectors of $\Delta$ are integers or $\pm \infty$.
This covers three different cases:
\begin{enumerate}
\item $d\ge 0$, $m=1$, $c=0$.  In this case $S(\Delta_{0,1,d})=\P^2$, $L(\Delta_{0,1,d})=dH$, with $H$ the hyperplane bundle on $\P^2$.
\item $d\ge 0$, $m\ge 1$, $c=0$. In this case  $S(\Delta_{0,m,d})=\P(1,1,m)$, $L(\Delta_{0,m,d})=dH$, with $H$ the hyperplane bundle on $\P(1,1,m)$ with self intersection $m$.
\item $d\ge 0$, $\ge 0, m\ge 0,c\ge 0$. 
In this case $S(\Delta_{c,m,d})$ is the rational ruled surface $\Sigma_m$. Let $E$ be the class of a section with self intersection $-m$ and $F$ the class of a fibre.
Let $H:=E+mF$. Then $L(\Delta)=cF+dH$.
\end{enumerate}
Note that in some cases the same lattice polygon corresponds to different pairs of a surface and a line bundle, but by the above the refined Severi degree only depends on $\Delta$.
\end{example}
In \cite{BG} it was also shown that the refined Severi degrees can for $h$-transversal lattice polygons be computed in terms of Floor diagrams.
Here we will not recall the definition of the refined Severi degrees as a count of tropical curves, but directly review them in terms of Floor diagrams which are very closely related to long-edge graphs. We will also restrict our attention to the lattice polygons  $\Delta_{c,m,d}$ of  Example \ref{polygons}, and thus to $\P^2$, $\P(1,1,m)$ and $\Sigma_m$.
In the following we fix $d,m,c$ and write $\Delta=\Delta_{c,m,d}$.

\begin{defn}
A \emph{$\Delta$-floor diagram} $\D$ consists of:
\begin{enumerate}
\item A graph on a vertex set $\{1, \dots, d\}$, possibly with
  multiple edges, with edges directed $i \to j$ if $i < j$. Edges $e$ carry a weight $w(e)\in \Z_{>0}$.
\item A sequence $(s_1, \dots, s_{d})$ of non-negative integers such
  that $s_1 + \cdots + s_{d} =c$.
  \item (Divergence Condition) For each vertex $j$ of $\D$, we have 
\[
\text{div}(j) \stackrel{\text{def}}{=} \sum_{ \tiny
     \begin{array}{c}
  \text{edges }e\\
j \stackrel{e}{\to} k
     \end{array}
} w(e) -   \sum_{ \tiny
     \begin{array}{c}
  \text{edges }e\\
i \stackrel{e}{\to} j
     \end{array}
} w(e)\le m + s_j.
\]
\end{enumerate}
\end{defn}

\begin{nota}
For an integer $n$ we introduce the quantum number $[n]_y$ by  
$$[n]_y=\frac{y^{n/2}-y^{-n/2}}{y^{1/2}-y^{-1/2}}=y^{{n-1}/2}+y^{{n-3}/2}+\ldots+y^{{-n+3}/2}+y^{{-n+1}/2}.$$
\end{nota}
\begin{defn}
\label{def:refined_multiplicity_FD}
We define the \emph{refined multiplicity} $\mult(\D, y)$ of a floor
diagram $\D$ as
\[
\mult(\D, y) = \prod_{\text{edges }e} \left( [w(e)]_y\right)^2. 
\]
\end{defn}
By definition  $\mult(\D, y)$ is a Laurent polynomial in $y$ with positive integral
coefficients.

\begin{defn}
\label{def:marking}
A \emph{marking} of a floor diagram $\D$ is defined by the following
four step process

{\bf Step 1:} 
For each vertex $j$ of $\D$ create $s_j$ new indistinguishable vertices and
connect them to $j$ with new edges directed towards $j$.

{\bf Step 2:} For each vertex $j$ of $\D$ create $m+s_j- div(j)$ new
indistinguishable vertices and connect them to $j$ with new edges
directed away from $j$. This makes the divergence of vertex $j$ equal
to $m$.

{\bf Step 3:} Subdivide each edge of the original floor
diagram $\D$ into two
directed edges by introducing a new
vertex for each edge. The new edges inherit their weights and orientations. Denote the
resulting graph $\widetilde{\D}$.

{\bf Step 4:} Linearly order the vertices of $\widetilde{\D}$
extending the order of the vertices of the original floor
diagram $\D$ such that, as before, each edge is directed from a
smaller vertex to a larger vertex.

The extended graph $\widetilde{\D}$ together with the linear order on
its vertices is called  a marked floor diagram or
\emph{marking} of the  floor diagram $\D$.
\end{defn}
The \emph{cogenus} of a marked floor diagram $\widetilde \D$ is $\delta(\widetilde \D):=\#(\Delta\cap \Z^2)-1-k$, where $k$ is the total number of vertices of $\widetilde \D$
(this coincides with the cogenus of the tropical curve corresponding to $\widetilde \D$, see e.g. \cite[Def.~4.2]{BG2}).
We  count marked floor diagrams up to equivalence. Two markings
$\widetilde{\D}_1$, $\widetilde{\D}_2$ of a floor diagram $\D$ are \emph{equivalent} if there exists an 
automorphism of weighted graphs which preserves the vertices of $\D$ and maps $\widetilde{\D}_1$ to
$\widetilde{\D}_2$. We denote  $\nu(\D)$ the number of
markings  $\widetilde{\D}$ of $\D$ up to equivalence.
Denote by $\FD(\Delta, \delta)$ the set of $\Delta$-floor diagrams $\D$ with cogenus $\delta$.

\begin{thm} (\cite[Thm.~5.7]{BG})
\label{Refinedfloor}
For $\Delta=\Delta_{c,m,d}$ as in Example \ref{polygons}  and $\delta \ge
0$, we have
\[
N^{\Delta, \delta}(y) = \sum_{\D \in \FD(\Delta, \delta)} \mult(\D; y)
\cdot \nu(\D).
\]
\end{thm}

\subsection{Caporaso-Harris type recursion}
In \cite{BG} also a Caporaso-Harris type recursion is proven for the refined Severi degrees of $\P^2$, $\P(1,1,m)$ and $\Sigma_m$, thus showing that they coincide with the refined Severi degrees as
defined in \cite{GS}. This recursion can be easily programmed in Maple, and has been extensively used in the course of this paper to find conjectural generating functions for the refined Severi degrees.
In this section let $S$ be $\P^2$, $\P(1,1,m)$ and $\Sigma_m$. 
We first recall the notations.

By a {\it sequence} we mean a collection $\alpha=(\alpha_1,\alpha_2,\ldots)$ of nonnegative integers, almost all of which are zero. For two sequences
$\alpha$, $\beta$ we define 
$|\alpha|=\sum_i\alpha_i$, $I\alpha=\sum_i i\alpha_i$, 
$\alpha+\beta=(\alpha_1+\beta_1,\alpha_2+\beta_2,\ldots)$, and $\binom{\alpha}{\beta}=\prod_i \binom{\alpha_i}{\beta_i}$. We write $\alpha\le\beta$ to mean $\alpha_i\le \beta_i$ for all $i$.
We write $e_k$ for the sequence whose $k$-th element is $1$ and all other ones  $0$.
We usually omit trailing zeros. 
For sequences $\alpha$, $\beta$, and $\delta\ge 0$, let $\gamma(L, \beta,\delta)=\dim|L|-HL+|\beta|-\delta$.

The relative refined Severi degrees $N^{(S,L),\delta}(\alpha,\beta)(y)$ is defined in \cite[Def.~7.2]{BG}. Here $N^{(S,L),\delta}(\alpha,\beta)(1)$ is the  relative Severi degree, i.e. the number of $\delta$-nodal curves in $|L|$ not containing $H$, through $\gamma(L, \beta,\delta)$ general points,
and with $\alpha_k$ given points of contact of order $k$ with $H$, and $\beta_k$ arbitrary points of contact of order $k$ with $H$.
By definition the relative refined Severi degrees contain the refined Severi degrees as special case:
$N^{(S,L),\delta}(0,(LH))(y)=N^{(S,L),\delta}(y)$.

\begin{thm} (\cite[Thm.~7.5]{BG}) \label{Caporaso}
Let  $L$ be a line bundle on $S$ and let $\alpha$, $\beta$ be sequences with $I\alpha+I\beta=HL$, and 
let $\delta\ge 0$ be an integer. 
If $\gamma(L, \beta,\delta)>0$, then 
\begin{equation}\label{refrec}
\begin{split}
N^{(S,L),\delta}(\alpha,\beta)(y)&=\sum_{k:\beta_k>0} [k]_y \cdot  N^{(S,L),\delta}(\alpha+e_k,\beta-e_k)(y)\\
&+\sum_{\alpha',\beta',\delta'}
\left(\prod_i [i]_y^{\beta_i'-\beta_i}\right)
\binom{\alpha}{\alpha'}\binom{\beta'}{\beta}  N^{(S,L-H),\delta'}(\alpha',\beta')(y).
\end{split}
\end{equation}
Here the second sum runs through all  $\alpha',\beta',\delta'$ satisfying the condition
\begin{equation}\label{relcong}
\begin{split}
\alpha'&\le \alpha, \ \beta'\ge \beta,\ I\alpha'+I\beta'=H(L-H),\\ \delta'&=\delta+g(L-H)-g(L)+|\beta'-\beta|-1=\delta-H(L-H)+|\beta'-\beta|.
\end{split}
\end{equation}
{\bf Initial conditions:} if $\gamma(L,\beta,\delta)=0$ we have  $N^{(S,L),\delta}(\alpha,\beta)(y)=0$, except for $N^{(\P^2,H),0}((1),(0))(y)=1$, $N^{(\P(1,1,m),H),0}((1),(0))(y)=1$ and  
$N^{(\Sigma_m, kF),0}((k),(0))(y)=1$, for all $k\ge 0$.
\end{thm}

\subsection{Long edge graphs}

 We review long edge graphs from \cite{BCK}, \cite{L}, \cite{LO}, working in the context of refined invariants.
 They are very close related to Floor diagrams.
 We follow   the presentation in \cite{L}, \cite{LO}. The arguments used are similar to those of \cite{L}, \cite{LO}.
 \begin{defn}
 A {\em long edge graph} $G$ is a graph $(V,E)$ with a weight function $w:E\to \Z_{>0}$ satisfying the following.
 \begin{enumerate}
 \item
 The vertex set is $V=\Z_{\ge 0}$, the edge set $E$ is finite.
 \item $G$ can have multiple edges, but no loops.
 \item $G$ has no short edges, i.e. no edges connecting $i$ and $i+1$ of weight $1$.
 \end{enumerate}
 \end{defn}
 An edge connecting $i$ and $j$ with $i<j$ will be denoted $(i\to j)$ (note that there can be more than one such edge).
 The  {\em length} of an edge $e=(i\to j)$ is $\ell(e):=j-i$.
 
 \begin{defn}\label{MG}
 Given a long edge graph $G=(V,E,w)$, the {\em refined multiplicity} of $G$ is 
 $$M(G)(y):=\prod_{e\in E} ([w(e)]_y)^2.$$ The {\em Severi multiplicity} $m(G)$ and the {\em Welschinger multiplicity} of $G$ are
  $$m(G):=M(G)(1)=\prod_{e\in E} w(e)^2,\qquad r(G):=M(G)(-1)=\begin{cases} 1& \hbox{all $w(e)$ are odd},\\
  0&\hbox{otherwise}.\end{cases}$$
  The {\em cogenus} of $G$ is 
  $\delta(G):=\sum_{e\in E} (\ell(e)w(e)-1).$
 
 We denote $\minv(G)$ (resp.~ $\maxv(G)$)  the smallest (resp.~ largest) vertex $i$ of $G$ adjacent to an edge.
 The {\em length} of $G$ is $l(G):=\maxv(G)-\minv(G)$. 
 
 We denote $G_{(k)}$ the graph obtained by shifting all edges of $G$ to the right by $k$. 
 \end{defn}
 
 \begin{defn}
 Let $G$ be a long edge graph.
 For any $j\in \Z_{\ge 0}$ let $\lambda_j(G):=\sum_e w(e)$, for $e$ running through the edges $(i\to k)$ with $i<j\le k$.
 
 For $\beta=(\beta_0,\ldots,\beta_M)$ a sequence of nonnegative integers, $G$ is called {\em $\beta$-allowable} if $\maxv(G)\le M+1$ and $\beta_{j-1}\ge\lambda_j(G)$ for all $j=1,\ldots,M+1$. 
 $G$ is called {\em strictly $\beta$-allowable} if it is $\beta$-allowable and furthermore
 all edges incident to $0$ or $M+1$ have weight $1$. Also write $\overline\lambda_{j}(G):=\lambda_j(G)-\#\{\hbox{edges } (j-1\to j)\}$.
 $G$ is called {\em $\beta$-semiallowable} if $\maxv(G)\le M+1$ and $\beta_{j-1}\ge \overline \lambda_j(G)$ for all $j$.
 \end{defn}
 
  In this paper we will mostly consider the following sequences.
 \begin{nota}
 Let $c,d,m\in \Z_{\ge 0}$. We put 
 $s(c,m,d):=(e_0,\ldots,e_d)$ with $e_i=c+mi$.
 \end{nota}
 
\begin{defn}
A long edge graph $\Gamma$ is a {\em template} if for any vertex $1\le i\le \ell(\Gamma)-1$ there exists at least one edge
$(j\to k)$ with $j<i<k$. A long edge graph $G$ is called a {\em shifted template} if $G=\Gamma_{(k)}$ for some template $k\in \Z_{\ge 0}$.
\end{defn}

\begin{defn} Let $G$ be $\beta$-allowable for $\beta=(\beta_0,\ldots,\beta_M)$.
Define a new graph $\ext_\beta(G)$ by adding $\beta_{j-1}-\lambda_j(G)$ edges of weight $1$ connecting $j-1$ and $j$ for all $j=1,\ldots,M+1$.

A {\em $\beta$-extended ordering} of $G$ is a total ordering of the vertices and edges of $\ext_\beta(G)$, such that 
\begin{enumerate}
\item  it extends the natural ordering of the vertices $0,1,2,\ldots$, 
\item if an edge $e$ connects vertices $i$ and $j$, then $e$ is between $i$ and $j$.
\end{enumerate}
Two extended orderings $o,o'$ of $G$ are considered equivalent if there is an automorphism of the edges, permuting only edges connecting the same vertices and of the same weight which sends $o$ to $o'$.
\end{defn}
 
 \begin{defn}
 For a long edge graph let $P_\beta(G)$ be the number of $\beta$-extended orderings of $G$ up to equivalence. 
 Here $P_\beta(G)$ is defined to be $0$, if $G$ is not $\beta$-allowable. 
 Furthermore let $P^s_\beta(G):=\begin{cases} P_\beta(G)& G\hbox{ strictly $\beta$-allowable},\\
 0& \hbox{otherwise.}\end{cases}$
 \end{defn}
 
 \begin{defn} \label{Ndelbeta} Given $\beta\in \Z^{M+1}_{\ge 0}$, define 
 $$N^\delta_\beta(y):=\sum_{G}M(G)P^s_\beta(G), \quad n^\delta_\beta:=\sum_{G}m(G)P^s_\beta(G),\quad W^\delta_\beta:=\sum_{G}r(G)P^s_\beta(G),$$ 
 where the summation is over all long edge graphs $G$ of cogenus $\delta$. 
 \end{defn}
 
\begin{nota}
 We denote $\Sigma_m:=\P(\oo\oplus\oo(m))$ the $m$-th rational ruled surface. Let $F$ be the class of the fibre of the ruling and let $E$ be the class of a section with 
 $E^2=-m$. We denote $H:=E+mF$.
 \end{nota}
 
 The connection to the refined Severi and tropical Welschinger numbers is given by 
 \begin{thm}\label{severisequence}
  \begin{enumerate}
 \item For the refined Severi degrees of $\P^2$, $\P(1,1,m)$ and $\Sigma_m$ we have
$
 N^{d,\delta}(y)=N^\delta_{s(0,d,1)}(y)$, 
 $N^{(\P(1,1,m),dH),\delta}(y)=N^\delta_{s(0,m,d)}(y),$
 $N^{(\Sigma_m,cF+dH),\delta}(y)=N^\delta_{s(c,m,d)}(y).
$ 
 \item For the Severi degrees we have $
 n^{d,\delta}=n^\delta_{s(0,d,1)}$, 
 $n^{(\P(1,1,m),dH),\delta}=n^\delta_{s(0,m,d)},$
 $n^{(\Sigma_m,cF+dH),\delta}=n^\delta_{s(c,m,d)}.$
 \item For the  Welschinger numbers  we have 
 $W^{d,\delta}=W^\delta_{s(0,d,1)}$, 
 $W^{(\P(1,1,m),dH),\delta}=W^\delta_{s(0,m,d)},$
 $W^{(\Sigma_m,cF+dH),\delta}=W^\delta_{s(c,m,d)}.$
 \end{enumerate}
\end{thm}

\begin{pf} The proof is similar to that of \cite[Thm.~2.7]{BCK}, we include it for completeness.
It is enough to prove (1), because by \defref{Ndelbeta} and \defref{MG} we have $n^\delta_\beta=N^\delta_\beta(1)$ and $W^\delta_\beta=N^\delta_\beta(-1)$, and
we know $N^{(S,L),\delta}(1)=n^{(S,L),\delta}$, $N^{(S,L),\delta}(-1)=W^{(S,L),\delta}$ for any pair $(S,L)$ of toric surface and toric line bundle.
Furthermore it is enough to prove (1) in case $S=\Sigma_m$, because by \thmref{Refinedfloor} we have
$N^{(\P(1,1,m),dH),\delta}(y)=N^{(\Sigma(1,1,m),dH),\delta}(y)$.

Let $\Delta=\Delta_{c,m,d}$ for $c,m,d\in \Z_{\ge 0}$.  Let $\beta:=s(c,m,d)$. We will show that $N^\delta_\beta$ is equal to the right hand side of \thmref{Refinedfloor}, thus finishing the proof. 
First we show that there is a bijection between $\Delta$-floor diagrams and strictly $\beta$-allowable long-edge graphs which respects the cogenus, by showing that both are in bijection to another set of graphs, which for the moment we will call $\beta$-graphs. A $\beta$-graph is defined precisely like a long edge graph, except that we also allow for short edges 
$(i\to i+1)$ of weight $1$, and we require $\beta_{j-1}=\lambda_j(G)$ for $j=1,\ldots,d+1$, where as before  $\lambda_j(G)=\sum_{e}w(e)$, with $e$ running through the edges $(i\to k)$ with $i<j\le k$. 
By definition it is clear that the map $G\mapsto \ext_{\beta}(G))$ defines a bijection from the strictly $\beta$-allowable long-edge graphs to the  $\beta$-graphs, and 
the inverse is given by removing all short edges $(i\to i+1)$ of weight $1$ from a $\beta$-graph. 
We define the cogenus of a $\beta$-graph by
$\delta(G)=\sum_e (l(e)w(e)-1)$, with $e$ running over all edges of $G$. It is obvious that $\delta(G)=\delta(\ext_{\beta}(G))$.

If $\D$ is a $\Delta$-floor diagram, we first perform steps (1) and (2) in \defref{def:marking}. Then we identify all vertices we have created in step (1) to a vertex $0$, and we identify all vertices we have created in step (2) to a vertex $d+1$, in addition we add vertices $\Z_{\ge d+2}$ to the graph obtained this way. It is easy to see that in this way we get a $\beta$-graph $G(\D)$. Clearly the map $\D\mapsto G(\D)$ is injective, as all the steps are injective, and by definition is is also clear that it is surjective.
If $\widetilde \D$ is a marking of $\D$, then we see that the total number of vertices of $\widetilde \D$ is equal to $d+\#E$ where $E$ is the set of edges of $G(\D)$.
Defining $M(F):=\prod_e  [w(e)]_y^2$ with $e$ running through the edges of the $\beta$-graph $F$,  Definitions \ref{MG} and \ref{def:refined_multiplicity_FD}
imply $\mult(\D)=M(G(\D))$ for a floor diagram $\D$ and $M(G)=M(\ext_\beta(G))$ for a long edge graph $G$. 
From the  definitions we  also see that 
$$\delta(G(\D))=\sum_{e\in E} w(e)l(e)-\#E=\sum_{i=1}^d \lambda_i(G(\D))-\#E=\#(\Delta\cap \Z^2)-d-\#E-1=\delta(\widetilde \D).$$ 
Note that  the markings of the $\Delta$-floor diagram $\D$ are in bijection with  the number of diagrams obtained by putting one vertex on every 
edge of $G(\D)$ and ordering all the vertices of the new diagram, preserving the order of the vertices of $G(\D)$, and such that the vertex introduced on an edge $(i\to j)$ lies between $i$ and $j$. But this number clearly is the same as the number of linear orders on the union of the vertices and edges of $G(\D)$, again preserving the order of the vertices and and such that the edge
 $(i\to j)$ lies between $i$ and $j$. By definition this is just the number of $\beta$-extended orderings of the long edge graph corresponding to $\D$.
\end{pf}

\begin{rem}
More generally the methods of \cite{BG} will show (using also the notations from \cite{LO} ) the following refined version of \cite[Thm.~2.12]{LO} 
(see \cite[Rem.~5.8]{BG}).

\begin{enumerate}
\item
For any $\delta\ge 0$, any $h$-transversal lattice polygon 
the refined Severi degree  is 
$$
N^{\Delta,\delta}(y)=\sum_{(\mathbf l,\mathbf r)}N^{\delta-\delta(\mathbf l,\mathbf r)}_{\beta(d^t,\mathbf r-\mathbf l)}(y).
$$
Here the summation is over all reorderings $\mathbf l$ and $\mathbf r$ of the multisets of left and right directions of $\Delta$, 
satisfying $\delta(\mathbf l,\mathbf r)\le \delta$, $\beta(d^t,\mathbf r-\mathbf l)\in \Z^{M+1}_{\ge 0}$.
\item With the same index of summation we have 
$$
n^{\Delta,\delta}=\sum_{(\mathbf l,\mathbf r)}n^{\delta-\delta(\mathbf l,\mathbf r)}_{\beta(d^t,\mathbf r-\mathbf l)},
\quad
W^{\Delta,\delta}=\sum_{(\mathbf l,\mathbf r)}W^{\delta-\delta(\mathbf l,\mathbf r)}_{\beta(d^t,\mathbf r-\mathbf l)},
$$
\end{enumerate}
\end{rem}

Following \cite{L},\cite{LO}, we consider logarithmic versions of $P_\beta(G)$ and $P^s_\beta(G)$, 
\begin{defn}
A {\em partition} of a long edge graph $G=(V,E,w)$ is a tuple $(G_1,\ldots,G_n)$ of nonempty long edge graphs such that the disjoint union of the (weighted) edge sets of 
$G_1, \ldots,G_n$ is the (weighted) edge set of $G$.

For any long edge graph define
\begin{align*}
\Phi_\beta(G)&:=\sum_{n\ge 1} \frac{(-1)^{n+1}}{n}\sum_{G_1,\ldots,G_n}\prod_{j=1}^nP_\beta(G_j),\\
\Phi^s_\beta(G)&:=\sum_{n\ge 1} \frac{(-1)^{n+1}}{n}\sum_{G_1,\ldots,G_n}\prod_{j=1}^nP^s_\beta(G_j),
\end{align*}
where both summations are over the partitions of $G$.
\end{defn}

Let $${\mathcal N}(\beta,y,t):=1+\sum_{\delta>0} N_\beta^\delta (y)t^\delta,\quad {\mathcal Q}(\beta,y,t):=\log({\mathcal N}(\beta,y,t))=\sum_{\delta>0} Q^\delta_\beta(y) t^\delta.$$
Then the same arguments as in \cite{LO} show that 
\begin{equation}\label{QQdel}
Q^\delta_\beta(y)=\sum_{G} M(G)\Phi^s_\beta(G),
\end{equation}
where the summation is again over all long-edge graphs of cogenus $\delta$.

\begin{defn}
Let $G$ be a long edge graph. 
Let  $\epsilon_0(G):=1$,  if all edges adjacent to $\minv(G)$ have weight $1$, and  $\epsilon_0(G):=0$ otherwise. Similarly let
$\epsilon_1(G):=1$,  if all edges adjacent to $\maxv(G)$ have weight $1$, and  $\epsilon_1(G):=0$ otherwise.
\end{defn}

By \cite[Lem.~2.15]{L} we have $\Phi^s_\beta(G)=0$, if $G$ is not a shifted template. On the other hand   \cite[Cor.~3.5]{L} says that for a template $\Gamma$ we have
$$
\Phi^s_\beta(\Gamma_{(k)})=
  \begin{cases}
    \Phi_\beta(\Gamma_{(k)})& 1-\epsilon_0(\Gamma)\le k\le M+\epsilon_1(\Gamma)-\ell(\Gamma)\\ 
    0 &\hbox{otherwise.}
  \end{cases}
$$
Together with \eqref{QQdel}, this gives the following refined version of \cite[Cor.~3.6]{LO}.

\begin{cor}\label{templatesum} Let $\beta=(\beta_0,\ldots,\beta_{M})\in \Z_{\ge 0}^{M+1}$. Then 
$$Q^\delta_\beta(y)=\sum_{\Gamma} M(\Gamma)\sum_{k=1-\epsilon_0(\Gamma)}^{M-\ell(\Gamma)+\epsilon_1(\Gamma)}\Phi_\beta(\Gamma_{(k)}),$$
where the first sum runs over all templates $\Gamma$ of cogenus $\delta$.
\end{cor}

\begin{thm}\cite[Thm.~3.8]{LO} \label{linear} Let $G$ be a long edge graph. 
There exists a linear multivariate function $\Phi(G,\beta)$ in $\beta$, such that for any $\beta$ such that $G$ is $\beta$-semiallowable, we have 
$\Phi_\beta(G)=\Phi(G,\beta)$. Furthermore writing $\beta=(\beta_0,\ldots,\beta_M)\in \Z_{\ge 0}^{M+1}$, the linear function
$\Phi(G,\beta)$ is a linear combination of the $\beta_{i}$ with $\minv(G)\le i\le \maxv(G)$.
\end{thm}


\section{Multiplicativity theorems}
In this section we will show that the generating functions for the refined Severi degrees on weighted projective spaces and rational ruled surfaces are multiplicative.

\begin{thm}\label{mainthm}
\begin{enumerate}
\item 
Let $c\ge \delta$ and $d\ge \delta$, then 
$Q^{(\Sigma_m,cF+dH),\delta}(y)$ is a $\Q[y^{\pm1}]$-linear combination of 
$1$, $c$, $d$, $cd$, $m$, $md$, $md^2$.
\item In particular if   $c\ge \delta$, $d\ge \delta$, then $Q^{(\P^1\times\P^1,cF+dH),\delta}(y)$ is a $\Q[y^{\pm1}]$-linear combination of 
$1$, $c+d$, $cd$.
\item 
Fix $m\ge 0$, $c\ge 0$. If $d\ge \delta$ then $Q^{(\Sigma_m, dH+cF),\delta}(y)$ is a polynomial of degree $2$ in $d$.
\item Fix $m\ge 0$. If  $d\ge \delta$, then 
$Q^{(\P(1,1,m), dH),\delta}(y)$ is a polynomial of degree $2$ in $d$. In particular for $d\ge \delta$, $Q^{d,\delta}(y)$ is a polynomial of degree $2$ in $d$.
\item 
If $d,m\ge \delta$, then $Q^{(\P(1,1,m),dH),\delta}(y)$ is a $\Q[y^{\pm1}]$-linear combination of $1$, $m$, $d$, $dm$, $d^2m$.
\end{enumerate}
\end{thm}

%
%
%

\begin{proof}
(1)
By \corref{templatesum} and \thmref{severisequence}, we have
\begin{equation} 
\label{qsum}Q^{(\Sigma_m,cF+dH),\delta}(y)=Q^{\delta}_{s(c,m,d)}(y)=\sum_{\Gamma} M(\Gamma)\sum_{k=1-\epsilon_0(\Gamma)}^{d-\ell(\Gamma)+\epsilon_1(\Gamma)}\Phi_{s(c,m,d)}(\Gamma_{(k)}),
\end{equation}
with $\Gamma$ running through all templates of cogenus $\delta$. 

Let $\Gamma$ now be a template of cogenus $\delta$, and let $k$ be an integer in $[1-\epsilon_0(\Gamma),d-\ell(\Gamma)+\epsilon_1(\Gamma)]$. Then by definition we get 
$\Phi_{s(c,m,d)}(\Gamma_{(k)})=\Phi_{s(c+km,m,\ell(\Gamma)-1)}(\Gamma).$
On the other hand by \cite[Lem.~4.2]{LO} we have $\overline \lambda_i(\Gamma)\le \delta$ for all $i$.
By our assumption we have $c\ge \delta\ge \overline\lambda_i(\Gamma)$, thus $\Gamma$ is $s(c+km,m,\ell(\Gamma)-1)$-semiallowable.
Therefore $\Phi_{s(c+km,m,\ell(\Gamma)-1)}(\Gamma)$ is a linear function in the $c+lm$, $k\le l\le k+\ell(\Gamma)-1$, thus it is linear function in $c$ and $km$ of the form $\alpha+\beta (c+ km) +\gamma m$, with $\alpha, \beta, \gamma\in \Q$.

Let $M_1:=d-\ell(\Gamma)+\epsilon_1(\Gamma)+\epsilon_0(\Gamma)$, $
M_2:=d-\ell(\Gamma)+\epsilon_1(\Gamma)-\epsilon_0(\Gamma)+1.$
It is easy to see (and was already used in \cite{L}) that for  a template $\Gamma$ of cogenus $\delta$ we have
$\ell(\Gamma)-\epsilon_1(\Gamma)\le\delta$, so, by our assumption $d\ge \delta$, we have $M_1\ge 0$.
Recall that for integers $b\ge a-1$ we have the trivial identity  
$$\sum_{k=a}^b k=\frac{(a+b)(b-a+1)}{2}.$$
Thus we get
\begin{align*}
\sum_{k=1-\epsilon_0(\Gamma)}^{d-\ell(\Gamma)+\epsilon_1(\Gamma)}\Phi_{s(c,m,d)}(\Gamma_{(k)})  & = \sum_{k=1-\epsilon_0(\Gamma)}^{d-\ell(\Gamma)+\epsilon_1(\Gamma)}\big(\alpha+\beta (c+ km) +\gamma m\big) \\
  & = M_1(\alpha+\beta c+\gamma m)+\frac{M_1M_2}{2} \beta m, 
\end{align*}
which is a  $\Q$-linear combination of $1,c,d,cd,m,md,md^2$. Thus the claim follows by \eqref{qsum}.

(2) By (1) $Q^{(\P^1\times\P^1,cF+dH),\delta}(y)$ is a linear combination of $1$, $c$, $d$, $cd$. It is clearly symmetric under exchange of $c$ and $d$, and thus a linear combination of $1$, $c+d$, $cd$.

(3) 
By \corref{templatesum} and \thmref{severisequence}, 
\begin{equation}\label{qsum1}
Q^{(\Sigma_m,cF+dH),\delta}(y)=Q^{\delta}_{s(c,m,d)}(y)=\sum_{\Gamma} M(\Gamma)\sum_{k=1-\epsilon_o(\Gamma)}^{d-\ell(\Gamma)+\epsilon_1(\Gamma)}\Phi_{s(c,m,d)}(\Gamma_{(k)}),
\end{equation}
with $\Gamma$ running through all templates of cogenus $\delta$. 

Let $\Gamma$ be a template of cogenus $\delta$, and let $k$ be an integer in $[1-\epsilon_0(\Gamma),d-\ell(\Gamma)+\epsilon_1(\Gamma)]$.
Then by definition we get 
$\Phi_{s(c,m,d)}(\Gamma_{(k)})=\Phi_{s(c+km,m,\ell(\Gamma)-1)}(\Gamma).$ 
For a rational number $a$ we denote by $\lceil a\rceil$ the smallest integer bigger or equal to $a$.
We put $$k_{min}:=\max\left.\left(1,\max\left(\left\lceil \frac{\overline \lambda_i(\Gamma)}{m}\right\rceil-i+1\right| i=1,\ldots, \ell(\Gamma)\right)\right).$$
For $k\ge k_{min}$ we have that $(k+i-1)m+c\ge \overline \lambda_i(\Gamma)$ for all $i$, thus $\Gamma$ is 
$s(c+km,m,\ell(\Gamma)-1)$-semiallowable.  Thus for $k\ge k_{min}$, we have that $\Phi_{s(c+km,m,\ell(\Gamma)-1)}(\Gamma)$ is a linear function in the 
$lm$, $k\le l\le k+\ell(\Gamma)-1$, thus it is a linear function $\alpha+\beta km +\gamma m$, with $\alpha, \beta,\gamma\in \Q$.

By \cite[Lem.~4.2]{LO}, we have $\overline\lambda_i(\Gamma)\le\delta-\ell(\Gamma)+i+\epsilon_1(\Gamma)$.
As $\overline\lambda_i(\Gamma)\ge 0$, this implies 
$$\left\lceil \frac{\overline \lambda_i(\Gamma)}{m}\right\rceil-i+1\le \delta+\epsilon_1(\Gamma)-\ell(\Gamma)+1$$ for all $i$.
By the inequality $\ell(\Gamma)-\epsilon_1(\Gamma)\le\delta$, already used in part (1), this implies
 $k_{min}\le \delta+\epsilon_1(\Gamma)-\ell(\Gamma)+1$. By our assumption $d\ge \delta$, we have $d-\ell(\Gamma)+\epsilon_1(\Gamma) -k_{min}+1\ge 0$.
Therefore the same argument as in (1) shows that the sum
$$\sigma(\Gamma,k_{min}):=\sum_{k=k_{min}}^{d-\ell(\Gamma)+\epsilon_1(\Gamma)}\Phi_{s(c,m,d)}(\Gamma_{(k)})$$ is a $\Q$-linear combination of 
$1$, $d$, $m$, $md$, $md^2$.
If we fix $m$, it is a linear combination of $1$, $d$, $d^2$.
But
$$\sum_{k=1-\epsilon_0(\Gamma)}^{d-\ell(\Gamma)+\epsilon_1(\Gamma)}\Phi_{s(c+km,m,l(\Gamma)-1)}(\Gamma)=
\sigma(\Gamma,k_{min})+\sum_{k=1-\epsilon_0(\Gamma)}^{k_{min}-1}\Phi_{s(c+km,m,l(\Gamma)-1)}(\Gamma).$$
The second sum is for fixed $m$ just a finite number, thus the claim follows.

(4)  As $Q^{(\P(1,1,m),dH),\delta}(y)=Q^{(\Sigma_m,dH),\delta}(y)$, (4) is a special case of (3).

(5)
By \corref{templatesum} and \thmref{severisequence}, 
\begin{equation} \label{qsum2}
Q^{(\P(1,1,m),dH),\delta}(y)=Q^{\delta}_{s(0,m,d)}(y)=\sum_{\Gamma} M(\Gamma)\sum_{k=1}^{d-\ell(\Gamma)+\epsilon_1(\Gamma)}\Phi_{s(0,m,d)}(\Gamma_{(k)}),
\end{equation}
with $\Gamma$ running through all templates of cogenus $\delta$. 
According to \corref{templatesum}, the inner sum starts at $k=1-\epsilon_0(\Gamma)$. But $\Gamma$ is  a template  and therefore not $(0,m,d)$-semiallowable. Thus (in case $\epsilon_0(\Gamma)=1$), the contribution for $k=0$  vanishes.

We have $Q^{(\P(1,1,m),dH),\delta}(y)=Q^{\delta}_{s(0,m,d)}(y)$, which is computed by the case $c=0$ of \eqref{qsum2}. 
If $m\ge \delta$, then $k_{min}=1$ for all templates $\Gamma$ of cogenus $\delta$, thus 
$$Q^{(\P(1,1,m),dH),\delta}(y)=Q^{\delta}_{s(0,m,d)}(y)=\sum_{\Gamma} M(\Gamma)\sigma(\Gamma,1),$$
with $\Gamma$ again running through the templates of cogenus $\delta$. By (3) this is a $Q[y^{\pm 1}]$-linear combination of $1$, $d$, $m$, $md$, $md^2$.
\end{proof}

\section{Relation to the conjectural generating functions of the refined invariants}
In \cite{GS} refined invariants $\widetilde N^{(S,L),\delta}(y)$ of pairs $(S,L)$ of a smooth projective surface and a line bundle on $S$ were introduced. 
These are symmetric Laurent polynomials in a variable $y$, whose coefficients can be expressed universally (independent of $S$ and $L$) as polynomials in the
four intersection numbers $L^2$, $LK_S$ $K_S^2$ and $c_2(S)$ on the surface. 
For toric surfaces $S$ and sufficiently ample line bundles $L$  the  refined invariants $\widetilde N^{(S,L),\delta}(y)$ and  refined Severi degrees $N^{(S,L),\delta}(y)$ are conjectured to agree
 (\cite[Conj.~80]{GS}).

\begin{conj}\label{torconj}
Let $(S,L)$ be a pair of a smooth toric surface and a line bundle on $L$.
\begin{enumerate}
\item 
If $L$ is $\delta$-very ample on $S$, then $\widetilde N^{(S,L),\delta}(y)=N^{(S,L),\delta}(y)$.
\item 
$\widetilde N^{d,\delta}(y)=N^{d,\delta}(y)$ for $\delta\le 2d-2$.
\item $\widetilde N^{(\P^1\times\P^1,dH+cF),\delta}(y)=
N^{(\P^1\times\P^1,dH+cF),\delta}(y)$ for $\delta\le \min(2d,2c).$
\item $\widetilde N^{(\Sigma_m,dH+cF),\delta}(y)=
N^{(\Sigma_m,dH+cF),\delta}(y)$ for $\delta \le \min(2d,c).$
\end{enumerate}
\end{conj}

\begin{rem}\label{refinvnodepol}
Note that by definition, if \conjref{torconj} is true, then $N_\delta(d;y)=\widetilde N^{d,\delta}(y)$ for all $d$, $\delta$, and
$N_\delta(\Sigma_m,dH+cF)(y)=\widetilde N^{(\Sigma_m,dH+cF),\delta}(y)$ for all $m,d,c,\delta$. 
This is because both sides are polynomials in $d$ (respectively $d,c$) with coefficients in $\Q[y]$, which coincide
for all sufficiently large $d$ (respectively for all sufficiently large $d,c$).
\end{rem}

In \cite[Conj.~67]{GS} also a generating function for the refined invariants $\widetilde N^{(S,L),\delta}(y)$ is conjectured (and thus by \remref{refinvnodepol} for the 
$N_\delta(d;y)$ and the $N_\delta((\Sigma_m,cF+dH);y)$) \cite[Conj.67]{GS}. We list a number of equivalent formulations. 

\begin{nota}
We start by introducing some notations about quasimodular forms and theta functions, and reviewing some standard facts, which we will use throughout the paper. 
Modular forms depend on a variable $\tau$ in the complex upper half plane, and have a Fourier development in terms of $q:=e^{2\pi i \tau}$. We will write them 
as functions $f(q)$, because we are only interested in the coefficients of their Fourier development. Similarly theta functions will be written as functions 
$g(y,q)$, for $y=e^{2\pi i z}$, with $z\in \C$ and $q=e^{2\pi i \tau}$.
The Eisenstein series 
$$G_{2k}(q)=-\frac{B_{2k}}{4k}+\sum_{n>0}\sum_{d|n} d^{2k-1} q^k$$ are for $2k\ge 4$ modular forms of weight $2k$ on $SL_2(\Z)$, whereas
$G_2(q)$ is only a quasimodular form of weight $2$ on $SL_2(\Z)$. 
The Dirichlet $\eta$-function and the discriminant $\Delta(q)$ are
$$\eta(q):=q^{1/24}\prod_{n>0}(1-q^n),\qquad \Delta(q)=\eta(q)^{24}=q\prod_{n>0}(1-q^n)^{24}.$$
 The discriminant is a cusp form of weight $12$ on $SL_2(\Z)$.
The operator $D:=q\frac{\partial}{\partial q}$ sends (quasi)modular forms of weight $2k$ to quasimodular forms of weight $2k+2$.
We denote two of the standard theta functions by
\begin{align*}
\theta(y)=\theta(y,q)&:=\sum_{n\in\Z} (-1)^n q^{\frac{1}{2}(n+\frac{1}{2})^2}y^{n+\frac{1}{2}}=q^{\frac{1}{8}}(y^{\frac{1}{2}}-y^{-\frac{1}{2}})\prod_{n>0} (1-q^n)(1-q^ny)(1-q^n/y),\\
\theta_2(y,q)&:=\sum_{n\in \Z} (-1)^nq^{n^2/2}y^n,
\end{align*}
and the theta zero value
$\theta_2(q^2):=\theta_2(0,q^2)=\sum_{n\in \Z} (-1)^m q^{n^2}=\frac{\eta(q)^2}{\eta(q^2)}.$
In addition to $D:=q\frac{\partial}{\partial q}$ we also consider $\vphantom{a}'=y\frac{\partial}{\partial y}$. 
Let 
\begin{align*}
\widetilde\Delta(y,q)&:= \frac{\eta(q)^{18}\theta(y)^2}{y-2+y^{-1}} =q \prod_{n=1}^{\infty}(1-q^n)^{20}(1-yq^n)^{2}
(1-y^{-1}q^n)^2,\\
\widetilde{DG}_2(y,q)&:=\sum_{m= 1}^\infty  \sum_{d | m} \frac{m}{d}[d]_y^2q^{m}, \quad D\widetilde{DG}_2(y,q):=\sum_{m= 1}^\infty  \sum_{d | m} \frac{m^2}{d}[d]_y^2 q^{m} .
\end{align*}
\end{nota}

\begin{conj} 
\label{genfun} 
There exist universal power series 
$B_1(y,q)$, $B_2(y,q)$ in $\Q[y,y^{-1}]\lbr q\rbr $, such that for all pairs $(S,L)$ of a smooth projective surface and a line bundle on $L$, we have
\begin{equation}\label{genfun1}\sum_{\delta\ge 0} \widetilde N^{(S,L),\delta}(y) (\widetilde {DG}_2)^\delta=
\frac{(\widetilde{DG}_2/q)^{\chi(L)}B_1(y,q)^{K_S^2}B_2(y,q)^{LK_S}}{(\widetilde\Delta(y,q)\cdot D\widetilde{DG}_2(y,q)/q^2)^{\chi(\oo_S)/2}}
\end{equation}
\end{conj}
We  give two equivalent reformulations. 
$\widetilde {DG}_2$ as a power series in $q$ starts with $q$, let 
$$g(t):=g(y,t)=t + ((-y^2 - 4y - 1)/y)t^2 + ((y^4 + 14y^3 + 30y^2 + 14y + 1)/y^2)t^3+O(t^4)$$
be its compositional inverse. Write $g'(t):=\frac{\partial g}{\partial t}.$
\begin{rem}\label{reform}
Let $R\in \Q[y^{\pm 1}][[q]]$ be a formal power series. For polynomials $M^{(S,L),\delta}(y)\in \Q[y^{\pm 1}]$ the following three formulas are equivalent:
\begin{enumerate}
\item
$\displaystyle{
\sum_{\delta\ge 0} M^{(S,L),\delta}(y) (\widetilde {DG}_2)^\delta=
\frac{(\widetilde{DG}_2/q)^{\chi(L)}B_1(y,q)^{K_S^2}B_2(y,q)^{LK_S}}{(\widetilde\Delta(y,q)\cdot D\widetilde{DG}_2(y,q)/q^2)^{\chi(\oo_S)/2}}R(y,q)}$
\item $\displaystyle{\sum_{\delta\ge 0} M^{(S,L),\delta}(y) t^\delta=\frac{(t/g(t))^{\chi(L)}B_1(y,g(t))^{K_S^2}}
{B_2(y,g(t))^{-LK_S}}\left(\frac{g(t)g'(t)}{\widetilde\Delta(y,g)}\right)^{\chi(\oo_S)/2}R(y,g(t)),}$
\item For all $\delta \ge 0$  \\
$\displaystyle{M^{(S,L),\delta}(y)=\Coeff_{q^{(L^2-LK_S)/2}}\left[\widetilde{DG}_2(y,q)^{\chi(L)-1-\delta}
\frac{B_1(y,q)^{K_S^2}B_2(y,q)^{LK_S}D\widetilde{DG}_2(y,q)}{(\widetilde\Delta(y,q)\cdot D\widetilde{DG}_2(y,q))^{\chi(\oo_S)/2}}R(y,q)\right]}$
\end{enumerate}
\end{rem}
\begin{pf}
(2) is equivalent to (1) by noting that $D\widetilde{DG}_2(y,g(t))=\frac{g(t)}{g'(t)}\frac{\partial \widetilde{DG_2}(y,g(t))}{\partial t}= \frac{g(t)}{g'(t)}.$

Let $A$ be a commutative ring, and let $f\in A[[q]]$, $g\in q+qA[[q]]$. Then we get by the residue formula that 
$$f(q)=\sum_{l=0}^\infty g(q)^l\Coeff_{q^0}\left[\frac{f(q)Dg(q)}{g(q)^{l+1}}\right].$$ 
Applying this with $g(q)=\widetilde{DG}_2$ shows that  (1) is equivalent to (3).
\end{pf}

Part (2) of \remref{reform} shows in particular that according to \conjref{genfun} the $N^{(S,L),\delta}(y)$ have a generating function of the form
\eqref{prodSL}.
\begin{rem}\label{genmultipl}

We will in the future use the formula (3) of \remref{reform}. Note that this also has the following interpretation. 
Write $$A^{(S,L)}(y,q):=
\frac{B_1(y,q)^{K_S^2}B_2(y,q)^{LK_S}D\widetilde{DG}_2(y,q)}{(\widetilde\Delta(y,q)\cdot D\widetilde{DG}_2(y,q))^{\chi(\oo_S)/2}}.$$ 
Then the refined count of  curves in $|L|$ with only nodes as singularities satisfying  $k$ general point conditions is 
$\Coeff_{q^{(L(L-K_S)/2}}[\widetilde{DG}_2(y,q)^k A^{(S,L)}(y,q)]$. Thus it seems natural to expect the following general principle:
To each condition $c$ that we can impose at points  of $S$ to curves $C$ in $|L|$ (e.g. $C$ passing through a point with given multiplicity),
or just to points in $S$, 
(e.g. $S$ having a singular point) 
there corresponds a power  series $L_c\in \Q[y^{\pm 1}][[q]]$, such that, for $L$ sufficiently ample,
the refined count of curves in $|L|$ on $S$ satisfying conditions $c_1,\ldots,c_n$ is
$\Coeff_{q^{(L(L-K_S)/2}}[A^{(S,L)}(y,q)\prod_{i=1}^n L_{c_i} ]$. According to this principle the power series corresponding to passing through a point of $S$ would be $\widetilde{DG}_2$. 
In the second half of this paper we will give a number of instances of this principle. 
\end{rem}

By \remref{refinvnodepol} for $\P^2$ and rational ruled surfaces the conjecture says in particular
\begin{align}\label{P2gen}
&N_\delta(d;y)=\Coeff_{q^{(d^2+3d)/2}}\left[\widetilde{DG}_2(y,q)^{d(d+3)/2-\delta}
\frac{B_1(y,q)^{9}}{B_2(y,q)^{3d}}\left(\frac{D\widetilde{DG}_2(y,q)}{\widetilde\Delta(y,q)}\right)^{1/2}\right]\\
\label{sigmangen}&N_\delta((\Sigma_m,cF+dH);y)=\\ \nonumber&\Coeff_{q^{(d+1)(c+1+md/2)-1}}\left[
\frac{\widetilde{DG}_2(y,q)^{(d+1)(c+1+md/2)-1-\delta}B_1(y,q)^{8}}{B_2(y,q)^{2c+(m+2)d}}\left(\frac{D\widetilde{DG}_2(y,q)}{\widetilde\Delta(y,q)}\right)^{1/2}\right]
\end{align}
With $B_1(y,q)$, $B_2(y,g)$ given below modulo $q^{18}$ we have the following corollary.
\begin{cor}\label{refpol}
\begin{enumerate}
\item The formula \eqref{P2gen} and \conjref{torconj}(2) are true for $\delta\le 17$.
\item In case $m=0$ the formula \eqref{sigmangen}  and  \conjref{torconj}(2) is true for $\delta\le 12$.
\item The  formula \eqref{sigmangen}  and \conjref{torconj}(3) are true  for all $m$ and $\delta\le 8$. 
\end{enumerate}
\end{cor}


\begin{proof} 
(1). Using the Caporaso-Harris recursion, we computed the 
$N^{d,\delta}(y)$ for $d\le 19$, $\delta\le 19$. This also computes the $Q^{d,\delta}$ for $d\le 19$, $\delta\le 19$. 
Part (4) of \thmref{mainthm} gives $Q^{d,\delta}=Q_\delta(d)$ for $d\ge \delta$. As $Q_\delta(d;y)$ is a polynomial of degree
$2$ in $d$, the computation above determines $Q_\delta(d;y)$ and thus the $N_\delta(y;d)$ for $\delta\le 17$, giving the  claim.

(2) and (3). Using again the Caporaso-Harris recursion we computed the 
$N^{(\P^1\times\P^1,cF+dH),\delta}(y)$ for $c,d\le 13,\delta\le 13$. Again this gives the $Q^{(\P^1\times\P^1,cF+dH),\delta}$ for $c,d\le 13$, $\delta\le 13$. 
By part (2) of \thmref{mainthm} we have that $Q^{(\P^1\times\P^1,cF+dH),\delta}=Q_\delta((\P^1\times\P^1,cF+dH);y)$ for $c,d\ge \delta$. As $Q_\delta((\P^1\times\P^1,cF+dH);y)$ is a polynomial of bidegree
$(1,1)$ in $c,d$, the computation above determines $Q_\delta(((\P^1\times\P^1,cF+dH);y)$ and thus the $N_\delta((\Sigma_0,cF+dH);y)$ for $\delta\le 12$.
As $Q_\delta((\Sigma_m,cF+dH);y)$ is a linear combination of $1$, $c$, $cd$, $m$, $md$, $md^2$, in order to prove (2) we only need to determine the coefficients of $m$, $md$, $md^2$.
For this we can restrict to the case $m=1$, 
We computed $N^{(\Sigma_1,cF+dH),\delta}(y)$ for $c,\le 9$, $d\le 10$. This determines the coefficients of $m$, $md$, $md^2$ of  $Q_\delta((\Sigma_m,cF+dH);y)$ for $\delta\le 8$, giving the claim. 
\end{proof}

We list the leading terms of $B_1(y,q)$ and $B_2(y,q)$, with omitted terms determined by symmetry.
 {\small\begin{align*}
&B_1(y,q)=1 - q - (y +3 + 1/y)q^2 
+ (y^2 + 10y + 17 +\ldots)q^3  
 - (18y^2 + 87y + 135 + \ldots)q^4\\&
+ (12y^3 + 210y^2 + 728y + 1061 +\ldots)q^5 - (2y^4 + 259y^3 + 2102y^2 + 5952y + 8236 + \ldots)q^6 \\&+ (162y^4 + 3606y^3  + 19668y^2 + 48317y + 64253+\ldots)q^7 
- (47y^{5} + 3789y^4 + 41999y^3+ 177800y^2\\&+ 392361y + 505678+\ldots )q^8  
+ (5y^{6} + 2416y^{5}  + 60202y^{4} + 445989y^3 + 1576410y^2 + 3197831y 
\\&+ 4018919+\ldots)q^9 -
(896y^{6} + 58504y^{5} + 793194y^{4} + 4483755y^3 + 13818256y^2 +26192369y\\& + 32243357+\ldots)q^{10}
 + (176y^7 + 38236y^6+ 1017512y^5 + 9382867y^4 + 43520558y^3 + 120325637y^2 \\& + 215688799y+ 260959201 + \ldots)q^{11} 
 - (14y^8 + 16393y^7 + 944954y^6 + 14738959y^5 
+103623419y^4 \\&+ 412518547y^3 + 1043940859y^2 + 1785764779y + 2129062780+\ldots)q^{12} 
+ (4384y^8 + 631224y^7\\& + 17534642y^6 + 190488676y^5 + 1092093647y^4 + 3845977628y^3+ 9041155627y^2 + 14862430058y \\&+ 17497499443 +\ldots )q^{13}
 - (658y^9 + 298228y^8 + 15816382y^7 - 273455570y^6 + 2279829046y^5\\& + 11131917064y^4+ 35435770399y^3 + 78257451025y^2 + 124310761787y + 144758147754+\ldots )q^{14} 
 \\&+ (42y^{10} + 96604y^9 + 10758628y^8 + 308060184y^7 + 3800583626y^6 + 25834889754y^5 \\&+ 110712006552y^4 + 323710356925y^3 + 677516096371y^2 + 1044598390812y + 1204824660925+\ldots)q^{15}\\&
 -(20284y^{10} + 5452043y^9 + 272316274y^8 + 5094738491y^7+ 48707795806y^6+ 281165238614y^5 \\& + 1080786159810y^4 + 2938608835049y^3 + 5869829083826y^2 + 8816117002571y + 10082791437552+\ldots)q^{16} \\&
+ (2472y^{11} + 2015609y^{10} + 188032406y^9 + 5506997958y^8+ 75206548205y^7 + 588088410636y^6\\& + 2967196356618y^5 + 10400483736235y^4 + 26552849592007y^3 + 50907878544033y^2 + 74707191955540y \\&+ 84801344804750+\ldots)q^{17} +
O(q^{18}),\\
 &B_2(y,q)=\frac{1}{(1-yq)(1-q/y)}\big(1 + 3q - (3y + 1 + 3/y)q^2 + (y^2 + 8y + 18 +\ldots)q^3\\&- (13y^2+ 53y + 76 + \ldots)q^4 + (7y^3 + 100y^2 + 316y + 455 + \ldots)q^5- (y^4 + 112y^3 + 779y^2\\& +2076y + 2819 +\ldots)q^6 + (67y^4 + 1243y^3 + 6129y^2 + 14386y + 18870 +\ldots)q^7- (19y^{5} \\&+ 1281y^4 + 12417y^3 + 48879y^2 + 104034y + 132579 +\ldots)q^8 + (2y^{6} + 822y^{5} + 17542y^{4}\\& + 117829y^3 + 393703y^2 + 775411y + 965540 +\ldots )q^9 -(310y^{6} + 17206y^{5} + 207074y^{4}\\& + 1085712y^3 + 3197506y^2 + 5913778y + 7223539+\ldots )q^{10}
  +(62y^7 + 11505y^6 + 267658y^5 \\&+ 2249872y^4 + 9825927y^3 + 26163595y^2 + 45935572y + 55208836 + \ldots)q^{11} - (5y^8 + 5076y^7 \\&+ 253785y^6 +3555348y^5 + 23210920y^4 + 87929247y^3 + 215557414y^3 + 362229349y \\&+ 429395117 + \ldots)q^{12} + (1397y^8 + 174456y^7 + 4304488y^6 + 42877083y^5 + 231296838y^4 \\&+ 781220881y^3+ 1787129788y^2 + 2892830316y + 3388742192 + \ldots )q^{13} - (215y^9 + 85117y^8 \\&+ 3983060y^7 + 62465678y^6 + 484877903y^5+ 2249516882y^4 + 6909207376y^3+ 14901830113y^2 \\&+ 23353834274y + 27076007072+\ldots)q^{14} + (14y^{10} + 28472y^9 + 2793096y^8 + 71942817y^7 \\&+ 818536892y^6 + 5240193024y^5 + 21495922606y^4 + 60931593665y^3 + 124910088474y^2 \\&+ 190304808803y + 218642432495 + \ldots)q^{15} - (6158y^{10} + 1462435y^9 + 65354234y^8 \\&+ 1118442331y^7 + 9987960061y^6 + 54777796045y^5 + 202738958803y^4 + 536439701989y^3 \\&+ 1052049129591y^2 + 1563445962327y + 1781883877192+\ldots )q^{16} + (770y^{11} + 558612y^{10} \\&+ 46524657y^9 + 1238412474y^8 + 15681201140y^7 + 115681622517y^6 + 558367283967y^5 \\&+ 1893273288345y^4 + 4718572145488y^3 + 8899835406922y^2 + 12937087920811y \\&+ 14639451592197 +\ldots)q^{17} +O(q^{18})\big).\\
\end{align*}}

As noted above, the refined Severi degrees $N^{(S,L),\delta}(y)$ specialize at $y=1$ to the tropical Welschinger numbers $W^{(S,L),\delta}$. 
We specialize the above conjectures of \cite{GS} to the tropical Welschinger numbers. As the Caporaso-Harris recursion for the tropical Welschinger numbers is computationally much more efficient than that for the refined Severi degrees, the conjectures for the tropical Welschinger numbers can be proven for much higher 
$\delta$.

Let
$\eta(q):=q^{1/24}\prod_{n>0}(1-q^n)$  the Dirichlet eta function, $G_2(q):=-\frac{1}{24} +\sum_{n>0}\sum_{d|n} d q^n$ be the Eisenstein series,  and write $$
\overline G_2(q):=\widetilde{DG}_2(-1,q)=G_2(q)-G_2(q^2)
=\sum_{n>0} \left(\sum_{d|n, \ d \text{ odd}} \frac{n}{d} \right)q^n.$$
We note that $\widetilde{DG_2}(-1,q)=\overline G_2(q)$, and $\widetilde\Delta(-1,q)=\eta(q)^{16}\eta(q^2)^4$.
We write $\overline B_1(q):=B_1(-1,q)$, $\overline B_2(q):=B_2(-1,q)$.
 \conjref{torconj} specializes to the following (see also \cite{GS}).

\begin{conj}\label{GSPSigmaW}
\begin{align}\label{WP2gen}
&W_\delta(d)=\Coeff_{q^{(d^3+3d)/2}}\left[\overline G_2(q)^{d(d+3)/2-\delta}
\frac{\overline B_1(q)^{9}(D\overline G_2(q))^{1/2}}{\overline B_2(q)^{3d}\eta(q)^8\eta(q^2)^2}\right],\\
&\label{Wsigmangen}W_\delta((\Sigma_m,cF+dH))=\Coeff_{q^{((d+1)(c+1+md/2)-1}}\left[
\frac{\overline G_2(q)^{(d+1)(c+1+md/2)-1-\delta}\overline B_1(q)^{8}(D\overline G_2(q))^{1/2}}{\overline B_2(q)^{2c+(m+2)d}\eta(q)^8\eta(q^2)^2}\right]
\end{align}
\end{conj}

With $\overline B_1(q)$, $\overline B_2(q)$ given below modulo $q^{31}$ we have the following corollary.
\begin{cor}
\begin{enumerate}
\item The formula \eqref{WP2gen} is true for $\delta\le 30$. Furthermore for $\delta\le 30$ and $d\ge \delta/3+1$ we have $W^{d,\delta}=W_\delta(d)$.
\item On $\P^1\times \P^1$ the  formula \eqref{Wsigmangen} is true for  $\delta\le 20$. Furthermore for $\delta\le 20$ and $\delta\le \min(20,3c,3d)$, we have
$W^{(\P^1\times\P^1,cF+dH),\delta}=W_\delta(\P_1\times \P_1,cF+dH)$.
\item For $m>0$, the  formula \eqref{Wsigmangen} is true for $\delta\le 11$. Furthermore for $\delta\le \min(11,3d,c)$ we have 
$W^{(\Sigma_m,cF+dH),\delta}=W_\delta(\Sigma_m,cF+dH)$.
\end{enumerate}
\end{cor}

\begin{proof} 
(1) Using the Caporaso-Harris recursion, we computed to the 
$W^{d,\delta}$ for $d\le 32$, $\delta\le 33$. This also computes the $Q^{d,\delta}(-1)$ for $d\le 32$, $\delta\le 33$. 
The same argument as in the proof of \corref{refpol} shows (1).
Using again the Caporaso-Harris recursion we computed the 
$W^{(\P^1\times\P^1,cF+dH),\delta}$ for $c,d\le 21,\delta\le 22$, and  computed $W^{(\Sigma_1,cF+dH),\delta}(y)$ for $c,d,\delta\le 13$. The same argument as in the proof of \corref{refpol} gives (2) and (3).
\end{proof}

{\small \begin{align*}
\overline B_1&(q)=1 - q- q^2 - q^3 + 3q^4 + q^5 - 22q^6 + 67q^7 - 42q^8 - 319q^9 + 1207q^{10}- 1409q^{11} \\&- 3916q^{12} + 20871q^{13} - 34984q^{14}- 37195q^{15} + 343984q^{16} - 760804q^{17} 
- 81881q^{18} \\&+ 5390386q^{19} - 15355174q^{20} + 8697631q^{21} + 79048885q^{22} 
- 293748773q^{23} + 329255395q^{24} \\&+ 1041894580q^{25} - 5367429980q^{26} + 8780479642q^{27} + 10991380947q^{28} \\& - 93690763368q^{29} + 203324385877q^{30}+O(q^{31}),\\
\overline B_2&(q)=1 + q + 2q^2 - q^3 + 4q^4 + 2q^5 - 11q^6 + 24q^7 + 4q^8 - 122q^9 + 313q^{10} - 162q^{11} \\&- 1314q^{12}+ 4532q^{13} - 4746q^{14}- 13943q^{15} + 68000q^{16} - 105786q^{17} - 124968q^{18} \\&+ 1025182q^{19} - 2139668q^{20} - 443505q^{21} + 15157596q^{22} - 41007212q^{23} + 19514894q^{24} \\&+ 214218876q^{25} - 755331892q^{26} + 780656576q^{27} + 2776494907q^{28}  \\&- 13420432234q^{29} + 20749875130q^{30} +O(q^{31}).
\end{align*}
}

\section{Correction term for singularities}


In this section we want to extend the above results and conjectures  to surfaces with singularities.
This section is partially motivated by the paper \cite{LO}, where this question is studied for the non-refined invariants for toric surfaces with rational double points.
We have conjectured above and given  evidence that there exist generating functions for the refined node polynomials on smooth toric surfaces $S$,
of the form $A_1^{L^2}A_2^{LK_S}A_3^{K_S^2}A_4^{\chi(\oo_S)}$ for universal power series $A_i\in \Q[y^{\pm 1}][[q]]$. 
It seems natural to conjecture that this extends to singular surfaces in the following form:
for every analytic type of singularities $c$ there is a universal power series $F_c(y,q)$ and the generating function for a singular surface $S$ is 
$A_1^{L^2}A_2^{LK_S}A_3^{K_S^2}A_4^{\chi(\oo_S)}\prod_c F_c^{n_c}$, where $n_c$ is the number of singularities of $S$ of type $c$.
For the case of toric surfaces given by $h$-transversal lattice polygons with only rational double points this problem has been solved in \cite{LO} for the
(non-refined) Severi degrees.
 
We start out by formulating a  conjecture for general singular toric surfaces, and then give more precise results for specific singularities. For rational double points
we conjecture that somewhat surprisingly the power series $F_c(y,q)$ is independent of $y$. In particular this says that the correction factor for $A_n$-singularities, determined in \cite{LO} for the Severi degrees, is the same for the Severi degrees and the tropical Welschinger invariants. 


Now let $S$ be a normal toric surfaces. We want to formulate  a conjecture about the refined Severi degrees $N^{(S,L),\delta}(y)$.
Note that the tropical curves counted in $N^{(S,L),\delta}(y)$ are not required to pass through any of the singular points of $S$. 
One can also reformulate the same conjecture in terms of the minimal resolution of $S$, i.e. a resolution $\pi:\widehat S\to S$, which contains no $(-1)$ curves in the fibres of $\pi$.

\begin{conj}\label{singsurfconj}
For every analytic type of singularities $c$ there  are formal power series 
$F_c\in \Q[y^{\pm 1}][[q]]$, $\widehat F_c\in \Q[y^{\pm 1}][[q]]$ such that the following hold.
Let $(S,L)$ be a pair of a projective toric surface and a  toric line bundle on $S$. Let $\widehat S$ be a minimal toric resolution of $S$ and denote by $L$ also the pullback of $L$ to $\widehat S$. Define $N^{(\widehat S,L),\delta}(y):=N^{(S,L),\delta}(y)$. 
If  $L$ is $\delta$-very ample on $S$, then 
\begin{align}
\label{Fc}N^{(S,L),\delta}(y)&=\Coeff_{q^{L(L-K_S)/2}}\left[\frac{\widetilde{DG}_2(y,q)^{\chi(L)-1-\delta}B_1(y,q)^{K_S^2}}{B_2(y,q)^{-LK_S}}\left(\frac{D\widetilde{DG}_2(y,q)}{\widetilde \Delta(y,q)}\right)^{1/2}\prod_c F_c(y,q)^{n_c}\right],\\
\label{Fcbar}N^{(\widehat S,L),\delta}(y)&=\Coeff_{q^{L(L-K_{\widehat S})/2}}\left[\frac{\widetilde{DG}_2(y,q)^{\chi(L)-1-\delta}B_1(y,q)^{K_{\widehat S}^2}}{B_2(y,q)^{-LK_{\widehat S}}}\left(\frac{D\widetilde{DG}_2(y,q)}{\widetilde \Delta(y,q)}\right)^{1/2}\prod_c \widehat F_c(y,q)^{n_c}\right].
\end{align}
Here $c$ runs through the analytic types of singularities of $S$, and $n_c$ is the number of singularities of $S$ of type $c$.
\end{conj}
We can see that the two formulas formulas \eqref{Fc}, \eqref{Fcbar} are equivalent.
Note that $LK_S=LK_{\widehat S}$. On the other hand it is easy to see that $K_{\widehat S}^2=K_S^2-\sum_{c}n_c e_c$ where $e_c$ is a rational number depending only on the singularity type $c$. 
Thus the two formulas are equivalent, via the identification $$\widehat F_c(y,q)=F_c(y,q)B_1(y,q)^{e_c}.$$
It turns out that the power series $\widehat F_c(y,q)$ are usually simpler, so we will restrict our attention to them.
Note that for a rational double point $c$ we have $e_c=0$ and thus $F_c=\widehat F_c$.

We give a slightly more precise version of the conjecture for a weighted projective space $\P(1,1,m)$ and its minimal resolution $\Sigma_m$, and prove some special cases of it. In this case  the exceptional divisor is the section $E$ with self intersection $-m$. The weighted projective space $\P(1,1,m)$  has one singularity of type  $\frac{1}{m}(1,1)$, i.e. the cyclic quotient of $\C^2$ by the $m$-th roots of unity $\mu_{m}$ acting by $\epsilon (x,y)=(\epsilon x,\epsilon y)$.
We write  $c_m$ for this singularity. 
It is elementary to see that 
\begin{align*}&K_{\Sigma_m}=-2H+(m-2)F=-\frac{m+2}{m}H-\frac{m-2}{m}E, \quad K_{\P(1,1,m)}=-\frac{m+2}{m}H, \\ 
&e_{c_m}=\frac{(m-2)^2}{m}, \quad
K_{\Sigma_m}^2=8, \quad dH K_{\Sigma_m} =d(m+2), \quad \chi(\Sigma_m,dH)=(md+2)(d+1)/2.
\end{align*}

\begin{conj}\label{ruledblow}
If $\delta\le 2d-1$, then 
\begin{equation}
\begin{split}\label{rulform}N^{(\Sigma_m,dH),\delta}(y)&=\Coeff_{q^{\frac{m}{2}d^2+(\frac{m}{2}+1)d}}\left[\frac{\widetilde{DG}_2(y,q)^{\frac{m}{2}d^2+(\frac{m}{2}+1)d-\delta}B_1(y,q)^{8}}{B_2(y,q)^{d(m+2)}}\Big(\frac{D\widetilde{DG}_2(y,q)}{\widetilde \Delta(y,q)}\Big)^{1/2} \widehat F_{c_m}(y,q)\right].
\end{split}\end{equation}
Furthermore 
we have for $m\ge 2$
\begin{align*}\widehat F_{c_m}&=1-mq+((m-2)y+(m^2/2+3m/2-5)+(m-2)y^{-1})q^2\\&-((m^2+5m-14)y+(m^3+9m^2+44m-132)/6+(m^2+5m-14)y^{-1})q^3+O(q^4),
\end{align*}
and 
\begin{align*}
\widehat F_{c_2}&=\sum_{n\in \Z} (-1)^n  q^{n^2}=1-2q+2q^4-2q^9+\ldots,\\ \widehat F_{c_3}&=1 - 3q + (y + 4 + y^{-1})q^2 - (10y + 18 + 10y^{-1})q^3 + ((6y^2 + 70y + 115 + 70y^{-1} + 6y^{-2})q^4 \\&
\qquad-
((y^3 + 94y^2 + 473y + 721y+ 473y^{-1} +94y^{-2} +y^{-3})q^5 + O(q^6) 
\\ \widehat F_{c_4}&=1 - 4q + (2y + 9 + 2y^{-1})q^2 - (22y + 42+ 22y^{-1})q^3\\& \qquad+ ((14y^2 + 164y + 273 + 164y^{-1} + 14y^{-2})q^4 +O(q^5).\\ 
\end{align*}
\end{conj}

\begin{prop} \label{propsigman} Let $\delta_2=8$, $\delta_3=5$, $\delta_4=4$, $\delta_m=3$ for $m\ge 5$.
Then \eqref{rulform}  is correct for $m\ge 2$ and $\delta\le \min(\delta_m,d)$.
\end{prop}

\begin{proof}
Using the Caporaso Harris recursion we computed $N^{(\Sigma_m,dH),\delta}$ for $2\le m\le 4$, $\delta\le \delta_m$ and 
$d\le d_m$ with $d_2=10$, $d_3= 7$, $d_4= 6$. We find that in this range \eqref{rulform} holds for $\delta\le \min(2d-1,\delta_m).$
By part (3) of \thmref{mainthm} we have  that  $Q^{(\Sigma_m,dH,\delta)}$ is a polynomial of degree $2$ in $d$ for $d\ge \delta$. 
By the computation we know this polynomial in the following cases:  
$(m=2,\ \delta\le 8)$, $(m=3,\ \delta\le 5)$, $(m=4,\ \delta\le 4)$. This shows the result for $m=2,3,4$.
Finally by part (5) of \thmref{mainthm}  we have that  $Q^{(\Sigma_m,dH,\delta)}(y)$ is for 
$d,m\ge \delta$ a polynomial in $d$ and $m$ of degree $2$ in $d$ and $1$ in $m$. By the above we know this polynomial as a polynomial in $d$ for $\delta=0,1,2,3$ and $m=3,4$. This determines it and thus also $Q^{(\Sigma_m,dH,\delta)}(y)$ and therefore also $N^{(\Sigma_m,dH,\delta)}(y)$, for  $\delta=0,1,2,3$ and $d,m\ge \delta$. 
The result follows.\end{proof}

The non-refined Severi degrees for toric surfaces with only rational double points given by $h$ transversal lattice polygons  have been studied in \cite{LO}.
The only rational double points which can occur in this case are $A_n$ singularities. For such surfaces they  prove the analogue of \conjref{singsurfconj} for  $y=1$ with precise bounds.
Furthermore they show
$$F_{a_n}(1,q)=\frac{\eta(q)^{n+1}}{\eta(q^{n+1})}=\prod_{k>0} \frac{(1-q^k)^{n+1}}{1-q^{(n+1)k}},$$
where we denote $F_{a_n}(y,q)$ the power series $F_c(y,q)$ for $c$ an $A_n$ singularity.
We conjecture that the same result holds also for the refined Severi degrees with the $F_{a_n}(y,q)$ independent of $y$.

\begin{conj} \label{conjan}
Let $S$ be projective normal  toric surface with only rational double points, more precisely with $n_k$ singularities of type $A_k$ for all $k$ (with $n_k$ only nonzero for finitely many $k$).
If $L$ is $\delta$-very ample on $S$, then
$$N^{(S,L),\delta}(y)=\Coeff_{q^{L(L-K_S)/2}}\left[\widetilde{DG}_2(y,q)^{\chi(L)-1-\delta}\frac{B_1(y,q)^{K_{\widehat S}^2}}{B_2(y,q)^{-LK_{\widehat S}}}\left(\frac{D\widetilde{DG}_2(y,q)}{\widetilde \Delta(y,q)}\right)^{1/2} \prod_{k}\left(\frac{\eta(q)^{k+1}}{\eta(q^{k+1})}\right)^{n_k}.\right].$$
\end{conj}


\begin{rem}\begin{enumerate}
\item $\P(1,1,2)$ has an $A_1$ singularity, and as we saw $\Sigma_2$ is a resolution of $\P(1,1,2)$. 
It is standard that $\theta_{2}(2\tau)=\frac{\eta(\tau)^2}{\eta(2\tau)}$. Thus for $\P(1,1,2)$  \conjref{conjan} is a special case of \conjref{ruledblow}
and \propref{propsigman} gives evidence for it.
\item We also used a version of the Caporaso Harris recursion for $\P(1,2,3)$. With the line bundle $dH$ with $d$ small for $H$ the hyperplane bundle. $\P(1,2,3)$ has one $A_1$ and one $A_2$ 
singularity, also in this case \conjref{conjan} is confirmed in the realm considered.
\item
Note that the conjecture that the $F_{a_n}(y,q)$ are independent of $y$ says in particular that the correction factor for the $A_n$ singularities is the same
for Severi degrees and tropical Welschinger invariants.
\end{enumerate}
\end{rem}

We want to  generalise this conjecture in another direction. 
Let $S$ be a singular toric surface with singular points $p_1,\ldots,p_r$ and a minimal toric resolution $\widehat S$ with exceptional divisors $E_1,\ldots,E_r$.
Let $L$ be a toric line bundle on $S$. We
have seen that 
$N^{(\widehat S,L),\delta}(y)=N^{( S,L),\delta}(y)$ is a refined count of $\delta$-nodal curves on $S$, which are not required to pass through the singular locus of $S$.
In a similar way we can interpret $N^{(\widehat S,L-k_1E_1-\ldots -k_r E_r),\delta}(y)$ as a refined count of curves in $|L|$ on $S$ which pass through the singular points $p_i$ with multiplicity
$-k_i E_i^2$. This even makes sense if $L$ is only a class of Weil divisors on $S$, the $k_i$ are not necessarily integral but $L-k_1E_1-\ldots -k_r E_r$ is a Cartier divisor on 
$\widehat S$. In this case the curves we count on $S$ are Weil divisors.

Here we will consider this question only in the case that $S$ has only $A_1$ singularities.
Denote $\eta(q)=q^{1/24}\prod_{n>0} (1-q^n)$ the Dirichlet eta function. Let $\theta_{2}(q):=\sum_{n\in \Z} (-1)^nq^{n^2/2}$ be one of the standard theta functions. 
Recall the Jacobi triple product formula
$$\eta(q^2)^3=q^{1/4} \sum_{n\ge 0} (-1)^n (2n+1)q^{n(n+1)}.$$

We define functions $f_l(q)$,  for $l\in \Z_{\ge 0}$ by
\begin{equation} \label{fkk2}
\begin{split}f_{2k}(q)&=
\frac{(-1)^{k}}{(2k)!}\sum_{n\in \Z} (-1)^n\left(\prod_{i=0}^{k-1} (n^2-i^2) \right) q^{n^2}=
\frac{(-1)^{k}}{(2k)!}\left(\prod_{i=0}^{k-1} (D-i^2) \right)\theta_{2}(q^2)\\
f_{2k+1}(q)&=\frac{(-1)^{k}}{(2k+1)!}\sum_{n\in \Z} (-1)^n(n+1/2)\left(\prod_{i=0}^{k-1} ((n+1/2)^2-(i+1/2)^2) \right) q^{(n+1/2)^2}\\
&=\frac{(-1)^{k}}{(2k+1)!}
\left(\prod_{i=0}^{k-1} (D-(i+1/2)^2) \right)\eta(q^2)^3.
\end{split}
\end{equation}
Here as before we denote $D=q\frac{d}{dq}$.
In particular we have
$$f_0(q)=\sum_{n\in \Z} (-1)^n q^{n^2},\quad f_{1}(q)=\sum_{n\ge 0} (-1)^n (2n+1)q^{(n+1/2)^2},\quad 
f_2(q)=\sum_{n>0} (-1)^{n-1} n^2	q^{n^2}.$$

We write 
$N^{(S,L),\delta}_{[k_1,\ldots,k_nr]}:=N^{(\widehat S,L-k_1E_1-\ldots -k_r E_r),\delta}(y)$, to stress that we view it as a count of curves on $S$ with 
prescribed multiplicities at the $A_1$-singularities.

\begin{conj}\label{A1con}
Let $S$ be a toric surface with only $A_1$ singularities $p_1,\ldots, p_r$. Fix $k_1,\ldots,k_r\in \frac{1}{2}\Z_{\ge 0}$. Let $\delta\ge 0$.
Let $L$ be a Weil divisor on $S$, such that $L-\sum_i k_iE_i$ is a Cartier divisor on $\widehat S$, which is $\delta$-very ample on any irreducible curve in $\widehat S$ not contained in 
$E_1\cup \ldots\cup E_r$..
Then 
\begin{equation}\label{A1mult}N^{(S,L),\delta}_{[k_1,\ldots,k_r]}(y)=\Coeff_{q^{L(L-K_S)/2}}\left[\frac{\widetilde{DG}_2(y,q)^{\chi(L)-\sum_i k_i^2-1-\delta}B_1(y,q)^{K_{S}^2}}{B_2(y,q)^{LK_{S}}}\left(\frac
{D\widetilde{DG}_2(y,q)}{\widetilde\Delta(y,q)}\right)^{1/2}\prod_{i=1}^r f_{2k_i}(q)\right].
\end{equation}
\end{conj}
Thus we claim that the correction factors for points of multiplicity $k$ at $A_1$ singularities of $S$ are given by the quasimodular forms $f_k(q)$.

Equivalently we can look at the same question on the blowup $\widehat S$. 
Write $\widehat L:=L-k_1E_1-\ldots-k_rE_r$ and 
$$\overline f_k(q)=\frac{f_k(q)}{ q^{k^2/4}},\quad k\in \frac{1}{2}\Z_{\ge 0},$$ 
then (with the same assumptions)  \eqref{A1mult} is clearly equivalent to

\begin{equation}\label{A1bl}N^{(\widehat S,\widehat L),\delta}(y)=\Coeff_{q^{\widehat L(\widehat L-K_{\widehat S})/2}}\left[\frac{\widetilde{DG}_2(y,q)^{\chi(\widehat L)-1-\delta}B_1(y,q)^{K_{\widehat S}^2}}{B_2(y,q)^{\widehat LK_{\widehat S}}}\left(\frac
{D\widetilde{DG}_2(y,q)}{\widetilde\Delta(y,q)}\right)^{1/2}\prod_{i=1}^r \overline f_{2k_i}(q)\right].
\end{equation}
In other words, the correction factors for $\widehat L$ not being sufficiently ample on $\widehat S$ are the $\overline f_l(q)$.

\begin{rem}\label{fkkk} Under the assumptions of the conjecture,  if the $k_i$ are sufficiently large with respect to $\delta$, then 
$\widehat L$ will be $\delta$-very ample on  $\widehat S$. This  means by \conjref{genfun} that for large $l$  the correction factor $\overline f_{l}(q)$ should be $1$ modulo some high power of $q$. In fact we find the following.

For  $l\in \Z_{>0} $ we can rewrite
$$\overline f_l(q)=\sum_{m\ge 0} (-1)^m \frac{2m+l}{m+l}\binom{m+l}{l}q^{m(m+l)}.$$
In particular $\overline f_l(q)\equiv 1\mod q^{l+1}.$
\end{rem}
\begin{proof}
First we deal with the case $l$ even. Note that 
$$\prod_{i=0}^{k-1}(n^2-i^2)=n \prod_{i=-k-1}^{k-1}(n-i).$$ Thus we get for $k>0$
\begin{align*}\overline f_{2k}(q)=\frac{(-1)^{k}}{(2k)!}\sum_{n\in \Z} (-1)^n\prod_{i=0}^{k-1} (n^2-i^2) q^{n^2-k^2}=
\sum_{n\ge k} (-1)^{n-k}\frac{2n}{2k}\binom{n+k-1}{2k-1}q^{n^2-k^2},\end{align*}
where we also have used that  $\binom{n+k-1}{2k-1}=0$ for $n<k$. Finally put $m=n-k$, so that 
$\frac{2n}{2k}\binom{n+k-1}{2k-1}= \frac{2m+2k}{m+2k}\binom{m+2k}{2k}$ and $n^2-k^2=m(m+2k)$.

The case $l$ odd is similar. Note that 
$$\prod_{i=0}^{k-1}((n+1/2)^2-(i+1/2)^2)= \prod_{i=-k+1}^{k}(n-i).$$
Thus we get
\begin{align*}\overline f_{2k+1}(q)&=\frac{(-1)^{k}}{(2k+1)!}\sum_{n\ge 0} (-1)^n(2n+1)\left(\prod_{i=0}^{k-1} ((n+1/2)^2-(i+1/2)^2) \right) q^{(n+1/2)^2-(k+1/2)^2}
\\&=
\sum_{n\in\Z} (-1)^{n-k}\frac{2n+1}{2k+2}\binom{n+k}{2k}q^{(n+1/2)^2-(k+1/2)^2},\end{align*}
and put again $m:=n-k$.
\end{proof}
\begin{rem}
It is again remarkable that the correction factors $f_{k}(q)$ are independent of the variable $y$. In particular this means again that the correction factor is the same for the Severi degrees and for the tropical Welschinger number.
\end{rem}

We specialise the conjecture to  case that $S$ is the weighted projective space $\P(1,1,2)$ with the resolution $\Sigma_2$ with more precise bounds for the validity.
Note that 
$$\chi(\Sigma_2,dH-kE)=(d+1)^2-k^2,\quad (dH-kE)K_{\Sigma_2}=(dH-kE)(-2H)=-4d,\quad K_{\Sigma_2}^2=8.$$

\begin{conj}\label{blowk}
 Let $d,k\in \frac{1}{2}\Z$ with $d-k\in \Z$. Then for $\delta\le 2(d-k)+1$, we have
\begin{equation}\label{blowkk}
N^{(\Sigma_2,dH-kE),\delta}(y)=\Coeff_{q^{d^2+2d-k^2}}\left[\frac{\widetilde{DG}_2(y,q)^{d^2+2d-k^2-\delta}B_1(y,q)^{8}}{B_2(y,q)^{4d}}\left(\frac
{D\widetilde{DG}_2(y,q)}{\widetilde\Delta(y,q)}\right)^{1/2}\overline f_{2k}(q)\right].
\end{equation}
\end{conj}

\begin{prop}
\begin{enumerate}
\item \conjref{blowk} is true for all $d$, all $k\le 5$ and $\delta\le 4$. 
\item The equation \eqref{blowkk} holds for all $d,k\ge 0$ with $\delta\le d-k$ and $\delta\le 4$.
\end{enumerate}
\end{prop}

\begin{proof}
We use the Caporaso-Harris recursion to compute $N^{(\Sigma_2,dH+cF),\delta}(y)=N^{(\Sigma_2,(d+c/2)H-c/2E),\delta}(y)$ for $\delta\le 8$, $d\le 6$ and $c\le 5$.
We find in this realm that $N^{(\Sigma_2,(nH-kE),\delta}(y)$ is equal to the right hand side of \conjref{blowk} for $\delta\le 2(n-k)+1$.
By \thmref{mainthm} $Q^{(\Sigma_2,dH+cF),\delta}(y)$ is for fixed $c\ge 0$ and for $d\ge \delta$ a polynomial of degree $2$ in $d$. Thus the above computations
determine this polynomial for $\delta\le 4$, and $c\le 5$. On the other hand in dependence of $c$ and $d$  we have that $Q^{(\Sigma_2,dH+cF),\delta}(y)$ 
is for $c,d\ge \delta$ a polynomial in $c$ and $d$ of degree $2$ in $d$ and $1$ in $c$. By the above we know this polynomial as a polynomial in $d$ for $c=4$ and $c=5$.
Thus it is determined and the claim follows.
\end{proof}

\section{Counting curves with prescribed multiple points}

Let $S$ be a smooth projective surface, let $p_1,\ldots,p_r$ be general points on $S$, and let $\widehat S$ be the blowup of $S$ in the $p_i$ with exceptional divisors
$E_i$. Let $n_1,\ldots, n_r\in \Z_{\ge 1}$.
Let $L$ be a sufficiently ample line bundle on $S$, and denote by the same letter its pullback to $\widehat S$.
Note that $N^{(\widehat S,L-\sum_i n_i E_i),\delta}(1)$
counts the complex curves on $S$ in $|L|$ with points of multiplicity $n_i$ in $p_i$ which have in addition $\delta$ nodes and pass through $\dim(|L-\sum_i n_i E_i)|)-\delta$ general points of $S$.
If $L$ is sufficiently ample, then the multiple points at the $p_i$ impose $\sum_i \binom{n_i+1}{2}$ independent conditions on curves in $|L|$.
Furthermore we see that 
$$\chi(L-\sum_i n_i E_i)=\chi(L)-\sum_i \binom{n_i+1}{2}.$$

Now assume that $S$ is a smooth projective toric surface. 
Let the $p_i\in S$ be fixed points of the torus action, so that $\widehat S$ is again a toric surface and the exceptional divisors $E_i$ are torus-invariant divisors.
Then by the above we can view $N^{(\widehat S,L-\sum_i n_i E_i),\delta}(y)$ as a refined count of curves in $|L|$ on $S$ with points of multiplicity $n_i$ at $p_i$ for all  $i$
and in addition $\delta$ nodes which  
pass through $$\dim(|L|)-\delta-\sum_{i} \binom{n_i+1}{2}$$ general points on $S$.

\begin{nota}
We denote $N^{(S,L),\delta}_{n_1,\ldots,n_r}(y):=N^{(\widehat S,L-\sum_i n_i E_i),\delta}(y)$.
\end{nota}

For an Eisenstein series $G_{2k}(q)$, we denote 
$$\overline G_k(q):=G_k(q)-G_k(q^2)=\sum_{n> 0}\sum_{\genfrac{}{}{0pt}{}{d|n}{\frac{n}{d}\text{ odd}} }d^{2k-1} q^n.$$
We write again $D:=q\frac{\partial}{\partial q}$. Note that $D^lG_{2k}(q)$ and $D^l\overline G_{2k}(q)$ are quasimodular forms of weight $2k+2l$.

\begin{conj}\label{multcon}
For each $i\ge 1$ there exists a universal power series
$H_i\in \Q[y^{\pm 1}][[q]]$, such that, whenever $L$ be sufficiently ample with respect to $\delta$, $r$ and $n_1,\ldots,n_r$, we have
\begin{equation}
\begin{split}
&N^{(S,L),\delta}_{n_1,\ldots,n_r}(y)=\\&\Coeff_{q^{(L^2-LK_S)/2}}\left[\widetilde{DG}_2(y,q)^{\chi(L)-1-\delta-\sum_i \binom{n_i+1}{2}}
\frac{B_1(y,q)^{K_S^2}B_2(y,q)^{LK_S}D\widetilde{DG}_2(y,q)}{(\widetilde\Delta(y,q)\cdot D\widetilde{DG}_2(y,q))^{\chi(\oo_S)/2}}\prod_{i=1}^r H_{n_i}(y,q)\right].
\end{split}
\end{equation}
Furthermore we conjecture for all $m>0$ the following:
\begin{enumerate}
\item $H_m(y,q)$ can be expressed in terms of Jacobi theta functions and quasimodular forms.
\item $H_m(1,q)$ is a (usually non-homogeneous) polynomial in the $D^lG_{2k}(q)$ of weight $\le 4k$.
\item $H_m(-1,q)$ is a (usually non-homogeneous) polynomial in the $D^lG_{2k}(q), D^l\overline G_{2k}(q)$ of weight $\le 2k$.
\end{enumerate}
\end{conj}
For small $m$ we explicitly conjecture the following formulas:
\begin{enumerate}
\item For $m\le 2$ we conjecture
\begin{align*}H_1(y,q)&=\widetilde{DG}_2(y,q),\quad H_2(y,q)=\frac{F_1(y,q)}{(y^{1/2}-y^{-1/2})^4}+\frac{F_2(y,q)}{(y^{1/2}-y^{-1/2})^2(y-y^{-1})},
\end{align*} with 
\begin{align*}
F_1(y,q)&=\sum_{n>0}\sum_{d|n}\frac{1}{2}
\left(-\frac{n^3}{d^3}+\frac{n^2}{d}-\frac{n}{d}\right)(y^{d/2}-y^{-d/2})^2 q^n\\
F_2(y,q)&=\sum_{n>0}\sum_{d|n}\left(\frac{n^2}{d^2}-\frac{n}{2}\right)
\frac{y^{d}-y^{-d}}{y-y^{-1}}q^n.\end{align*}


\item For the specialisation at $y=1$ we conjecture the following (dropping the $q$ from the notation).
\begin{align*}
H_1(1)&=DG_2, \\ H_2(1)&=-\frac{1}{24}DG_2+ \frac{1}{6}D^2G_2 -\frac{1}{8}DG_4 -\frac{1}{24}D^3G_2+ \frac{1}{24}D^2G_4\\
H_3(1)&=\frac{DG_2}{90} -\frac{D^2G_2}{18}+ \frac{DG_4}{24}- \frac{13D^3G_2}{288} -\frac{73D^2G_4}{1440}
+\frac{DG_6}{120}
-\frac{D^4G_2}{144}+ \frac{13D^3G_4}{1440}\\&-\frac{D^2G_6}{480}+ \frac{D^5G_2}{2880} -\frac{D^4G_4}{2016}+\frac{D^3G_6}{6912}+\frac{\Delta}{241920}\\
  H_4(1)&=-\frac{9DG_2}{1120}+ \frac{7D^2G_2}{160} -\frac{21DG_4}{640} -\frac{1063D^3G_2}{23040}+ \frac{1207D^2G_4}{23040} -\frac{3DG_6}{320}+ \frac{79D^4G_2}{5760} \\&
-\frac{43D^3G_4}{2304}\+ \frac{149D^2G_6}{26880}-\frac{DG_8}{2688} -\frac{91D^5G_2}{69120}+\frac{95D^4G_4}{48384} -\frac{461D^3G_6}{645120}+ \frac{101D^2G_8}{1451520}\\& 
-\frac{11\Delta}{5806080}+ \frac{D^6G_2}{17280} -\frac{89D^5G_4}{967680}
 + \frac{D^4G_6}{25920} -\frac{D^3G_8}{207360}+ \frac{D\Delta}{2903040} -\frac{D^7G_2}{967680}\\&
 + \frac{D^6G_4}{580608} -\frac{D^5G_6}{1244160}+ \frac{D^4G_4}{8211456} -\frac{D^2\Delta}{84913920}+\frac{\Delta G_4}{864864}
\end{align*}
\item At $y=-1$ we conjecture
\begin{align*}
H_1(-1)&=\overline G_2(q),\\
H_2(-1)&=\frac{1}{8}\big(\overline G_2-D\overline G_2+\overline G_4-DG_2\big),\\
H_3(-1)&=\frac{1}{24}\overline G_2-\frac{1}{24}DG_2+\frac{7}{96}\overline G_4-\frac{7}{96}D\overline G_2+\frac{1}{2}
\overline G_2^3-\frac{1}{192}D\overline G_4-\frac{5}{64}G_4\overline G_2+\frac{1}{96}D^2G_2\\&-\frac{5}{1024}DG_4,\\
H_4(-1)&=\frac{3\overline G_2}{128}-\frac{5DG_2}{192} -\frac{67D\overline G_2}{1536}+\frac{67\overline G_4}{1536}+\frac{35D^2G_2}{2304} -\frac{247DG_4}{24576}
+\frac{55\overline G_2^3}{144} -\frac{55G_4\overline G_2}{1536}\\&
-\frac{11D\overline G_4}{4608}
+\frac{D^3G_2}{192}+ \frac{25D^2G_4}{6144}-\frac{7DG_6}{8192}+ \frac{11\overline G_2^4}{8}-\frac{13\overline G_2D^2G_2}{192} +\frac{35\overline G_2 DG_4}{512}\\&
 -\frac{21G_6\overline G_2}{1024}+\frac{D^2\overline G_4}{512}.
\end{align*}
\end{enumerate}

\begin{rem}
Part (1)  of \conjref{multcon} is not formulated in a very precise way.  We  want to illustrate the statement for $H_1(y,q)$ and $H_2(y,q)$, which we have conjecturally determined.
Writing $\widetilde{DG}_2(y,q)=\frac{F_0(y,q)}{y-2+y^{-1}}$ we have
\begin{align*}F_0(y,q)&=-\frac{D\theta(y)}{\theta(y)}-3G_2,\\
F_1(y,q)&=\frac{1}{2}\frac{(D\theta(y))^2}{\theta(y)^2}+3\frac{D\theta(y)}{\theta(y)}G_2+\frac{1}{2}\frac{D\theta(y)}{\theta(y)}+\frac{15}{8}G_4-\frac{9}{4}DG_2+\frac{3}{2}G_2,\\
F_2(y,q)&=-\frac{1}{2}\frac{D\theta(y)\theta'(y)}{\theta(y)^2}-\frac{1}{6}\frac{D\theta'(y)}{\theta(y)}-2G_2\frac{\theta'(y)}{\theta(y)}.\\
\end{align*}
\end{rem}
\begin{pf} A similar computation has been done in \cite[Rem~1.4]{GS2}.
By definition we have
\begin{align*}
F_0(y,q)&=\sum_{m>0}\sum_{d>0} m(y^d-2+y^{-d})q^{md}=\sum_{md>0} my^dq^{md}-2G_2(q)+\frac{1}{12}.
\end{align*}
In \cite[page 456, compare (iii) and (vii)]{Z} it is proved that  
\begin{equation}\label{zagfun}
\frac{\theta'(0)\theta(wy)}{\theta(w)\theta(y)}=\frac{wy-1}{(w-1)(y-1)}-\sum_{nd>0}
\sgn(d)w^ny^dq^{nd}.
\end{equation}
Write $w=e^x$ and take the coefficient of $x$ on both sides of \eqref{zagfun}.
By the identity  \cite[eq.~(7)]{Z} we have
$$\frac{x\theta'(0)}{\theta(w)}=\exp\left(2\sum_{k\ge 2} G_k(q) \frac{z_1^k}{k!}\right).$$
This gives
$$\Coeff_{x}\left[\frac{\theta'(0)\theta(wy)}{\theta(w)\theta(y)}\right]=\Coeff_{x^2}\left[\frac{\theta(wy)}{\theta(y)}\right]+G_2(\tau)
=\frac{1}{2}\frac{\theta''(y)}{\theta(y)}+G_2(\tau)=\frac{D\theta(y)}{\theta(y)}+G_2(\tau),$$
where the last step is by the heat equation $\frac{1}{2}\theta''(y)=D\theta(y)$. 
On the other hand we compute
$$\Coeff_{z_1}\left[\frac{wy-1}{(w-1)(y-1)}-\sum_{nd>0}
\sgn(d)w^ny^dq^{nd}\right]=\frac{1}{12}-\sum_{nd>0} n y^{d}q^{nd}.$$
This proves the formula for $F_0$.

 We have 
\begin{align*}
F_2(y,q)&=\sum_{md>0} \sgn(d) (m^2-md/2)y^d)q^{md}.
\end{align*}
In \cite[Rem.~1.4]{GS2} it is shown (the statement there contains a misprint) that
$$\sum_{md>0} \sgn(d) m^2y^dq^{md}=-\frac{1}{\theta(y)}\left(\frac{2}{3}D\theta'(y)+2G_2(q)\theta'(y)\right).$$
We see by \eqref{zagfun} that 
\begin{align*} \sum_{md>0} \sgn(d) (-md/2)y^d)q^{md}=\frac{1}{2}D\left(\frac{\theta'(0)\theta(wy)}{\theta(w)\theta(y)}\Big|_{w=1}\right)=\frac{1}{2}D\left(\frac{\theta'(y)}{\theta(y)}\right).
\end{align*}
This shows the formula for $F_2$.

A similar but slightly more tedious computation shows the formula for $F_1$.
\end{pf}

The conjectural formulas of \conjref{multcon} were found by doing computations for $\P^2$ and its blowup $\Sigma_1$ with exceptional divisor $E$.
We use the Caporaso Harris recursion formula to compute 
$N^{(\Sigma_1,dH+mF),\delta}(y)=N^{(\Sigma_1,(d+m)H-mcE,\delta}$ for $d\le 11$, $m\le 4$ and $\delta\le 22$, in this realm the following conjecture is true.

\begin{conj}\label{P2blow}
There are power series $H_m(y,q)\in \Q[y^{\pm1}][[q]]$, such that the following holds.
For $d>0$, and $0\le m\le 4$ and $\delta\le 2d+1+m(m+1)/2$ we have
\begin{equation*}
\begin{split}
&N^{(\P^2,dH),\delta}_{m}(y)=\\&\Coeff_{q^{(d(d+3)/2}}\left[\widetilde{DG}_2(y,q)^{d(d+3)/2-m(m+1)/2-\delta}
\frac{B_1(y,q)^{9}(D\widetilde{DG}_2(y,q))^{1/2}}{B_2(y,q)^{-3d}\widetilde\Delta(y,q)^{1/2}}H_{m}(y,q)\right].
\end{split}
\end{equation*}
Furthermore $H_1(y,q)$, $H_2(y,q)$ coincide with the functions with the same name from \conjref{multcon},
and $H_i(1,q)$, $H_i(-1,q)$ coincide for $i=1,2,3,4$ with the $H_i(1)$, $H_i(-1)$ from \conjref{multcon}.
\end{conj}

\begin{prop} 
\conjref{P2blow} is true from $m\le 4$ and $\delta\le 9$.
\end{prop}

\begin{pf} The argument is the same as in several proofs before. By \thmref{mainthm} we get that $Q^{(\Sigma_1,dH+mF),\delta}$ is for $\delta\le d$ a polynomial of degree 
$2$ in $d$, which we know for $9\le d\le 11$. The result follows.
\end{pf}

Let $S$ be a toric surface and  $\widehat S$ be the blowup of $S$ in torus fixed point. Given $\delta$, if $m$ is sufficiently large and $L$ is sufficiently ample on $S$, then 
$L-mE$ will be sufficiently ample on $\widehat S$, so that \conjref{genfun} will apply to the pair $(\widehat S,L-mE)$:
\begin{align*}
&N^{(S,L),\delta)}_m(y)=N^{(\widehat S,L-mE),\delta}(y)=\\&\Coeff_{q^{(L^2-LK_S)/2-\binom{m+1}{2}}}
\left[\widetilde{DG}_2(y,q)^{\chi(L)-1-\delta-\binom{m+1}{2}}
\frac{B_1(y,q)^{K_S^2-1}B_2(y,q)^{LK_S+m}D\widetilde{DG}_2(y,q)}{(\widetilde\Delta(y,q)\cdot D\widetilde{DG}_2(y,q))^{\chi(\oo_S)/2}}\right].
\end{align*}
Combined with \conjref{multcon} this leads to the following conjecture.

\begin{conj} We have 
$$\frac{H_m(y,q)}{q^{\binom{m+1}{2}}}\equiv \frac{B_2(y,q)^m}{B_1(y,q)} \mod q^{m+1}.$$
\end{conj}

Thus, if  eventually one would find a way to explicitly determine the functions $H_m(y,q)$ for all $m$, this could give the unknown power series 
$B_1(y,q)$, $B_2(y,q)$ and thus complete the conjectural formulas of \cite{G},\cite{GS}.

It is natural to assume that the specialisation of \conjref{multcon} and also of the previous conjectures \conjref{singsurfconj}, \conjref{A1con} to $y=1$ hold for the usual Severi degrees $n^{(S,L),\delta}$ for projective algebraic surfaces, not just for toric surfaces.
Thus we get in particular the following generalisation of the original conjecture of \cite{G}.

Let $S$ be a projective algebraic surface with $A_1$-singularties $q_1,\ldots,q_s$. Let $p_1,\ldots p_r$ be distinct smooth points on $S$.
Let $m_1,\ldots,m_r\in \Z_{> 0}$, $n_1,\ldots,n_s\in \Z_{\ge  0}.$ Let $\widehat S$ be the blowup of $S$ in $q_1,\ldots,q_s,p_1,\ldots p_r$ and denote
$E_i$, $F_j$ the exceptional divisors over $q_i$, $p_j$ respectively. Let $L$ be a $\Q$-Cartier  Weil divisor on $S$, such that 
$\widehat L:=L-\sum_{i=1}^s m_i E_i-\sum_{i=1}^r n_i F_i$ is a Cartier divisor on $\widehat S$, which is $\delta$-very ample on all irreducible curves
in $\widehat S$ not contained in $E_1\cup\ldots\cup E_s\cup F_1\cup\ldots\cup F_r$.  
Denote $n^{(S,L),\delta}_{(m_1,\ldots,m_r),(n_1,\ldots,n_s)}:=n^{(\widehat S,\widehat L),\delta}$, which we could informally interpret as  the number of curves in $|L|$
which have multiplicity $m_i$ in $p_i$ and $n_j$ in $q_j$ for all $i,j$ and pass in addition through 
$$\dim|L|-\sum_{i=1}^r \binom{m_i+1}{2}-\sum_{j=1}^s \frac{n_j^2}{4}$$ general points on $S$, and have $\delta$ nodes as other singularities.

\begin{conj} \label{genconj} 
\begin{equation}
\begin{split}
n^{(S,L),\delta}_{(m_1,\ldots,m_r),(n_1,\ldots,n_s)}=\Coeff_{q^{(L^2-LK_S)/2}}
&\left[DG_2(q)^{\chi(L)-\sum_i \binom{m_i+1}{2}-\sum_{j} \frac{n_j^2}{4}-1}
\frac{B_1(q)^{K_S^2}B_2(q)^{LK_S}{D^2G}_2(q)}{(\Delta(q)\cdot {D^2G}_2(q))^{\chi(\oo_S)/2}}\right.\\
&\left.\left(\prod_{i=1}^r H_{n_i}(1,q)\right)\left(\prod_{i=1}^s f_{m_i}(q)\right)\right].
\end{split}
\end{equation}

\end{conj}


\end{document}